\documentclass[11pt,a4paper]{article}
\usepackage{geometry}

\UseRawInputEncoding
\usepackage{bbold}
\usepackage[utf8]{inputenc}
\usepackage[T1]{fontenc}
\usepackage{amsmath}
\usepackage{amsthm}
\usepackage{amssymb}
\usepackage{graphicx}
\usepackage{caption}
\usepackage{multirow}
\usepackage{mathrsfs}
\usepackage{tabularx}
\usepackage{dsfont}
\usepackage{hyperref}
\usepackage[dvipsnames]{xcolor}
\usepackage{verbatim}
\usepackage[shortlabels]{enumitem}
\hypersetup{colorlinks=true, linkcolor=blue!60!black, urlcolor=red!60!black, citecolor=blue!60!black}
\usepackage[ruled,vlined]{algorithm2e}
\SetKwProg{Init}{Initialisation}{}{}
\usepackage{float}
 \usepackage{ulem}
\usepackage{array,multirow,makecell}
\usepackage[justification=centering]{caption}
\setcellgapes{1pt}
\makegapedcells
\newcolumntype{R}[1]{>{\raggedleft\arraybackslash }b{#1}}
\newcolumntype{L}[1]{>{\raggedright\arraybackslash }b{#1}}
\newcolumntype{C}[1]{>{\centering\arraybackslash }b{#1}}
\makeatletter
\newcommand\justify{%
  \let\\\@centercr
  \rightskip\z@skip
  \leftskip\z@skip}
\makeatother

\newtheorem{thm}{Theorem}[section]
\newtheorem{coro}[thm]{Corollary}
\newtheorem{prop}[thm]{Proposition}
\newtheorem{lem}[thm]{Lemma}
\newtheorem{ass}{Assumption}[section]

\theoremstyle{definition}
\newtheorem{defi}{Definition}[section]
\newtheorem{exe}[thm]{Example}
\newtheorem{rmq}{Remark}[section]

\newenvironment{preuve}{\begin{proof} \rm}{\end{proof}}

\newtheorem*{prop*}{Proposition}
\newtheorem*{rmq*}{Remark}

\numberwithin{equation}{section}

\newcommand{\p}{\mathbb{P}}
\newcommand{\R}{\mathbb{R}}

\newcommand{\e}{\mathbb{E}}

\newcommand{\n}{\mathbb{N}}

\newcommand{\ind}{\mathds{1}}
\newcommand{\f}{\mathbb{F}}
\newcommand{\T}{\mathbb{T}}
\newcommand{\Rr}{\mathcal{R}}

\newcommand{\Hr}{\mathcal{H}}

\newcommand{\Y}{\mathcal{Y}}
\newcommand{\Sr}{\mathcal{S}}
\newcommand{\Z}{\mathcal{Z}}

\newcommand{\F}{\mathcal{F}}

\newcommand{\ABS}[1]{{\left| #1 \right|}} 
\newcommand{\PAR}[1]{{\left(#1\right)}} 
\newcommand{\SBRA}[1]{{\left[#1\right]}} 
\newcommand{\BRA}[1]{{\left\{#1\right\}}} 
\newcommand{\NRM}[1]{{\left\Vert #1\right\Vert}} 

\DeclareMathOperator{\Tr}{Tr}

\DeclareMathOperator{\argmin}{argmin}
\DeclareMathOperator{\dist}{dist}

\DeclareMathOperator{\Err}{Err}

\DeclareMathOperator{\loc}{loc}

\title{Deep learning scheme for forward utilities using ergodic BSDEs\thanks{Acknowledgements:  The authors research is part of the ANR project DREAMeS (ANR-21-CE46-0002) and benefited from the support of respectively the "Chair Risques Emergents en Assurance"  and  "Chair Impact de la Transition Climatique en Assurance"  under the aegis of Fondation du Risque, a joint initiative by  Risk and Insurance Institute of  Le Mans,  and MMA-Cov\'ea and Groupama respectively. }}  

\author{Guillaume Broux-Quemerais\textsuperscript{1} \and Sarah Kaaka\"i\textsuperscript{1,2} \and Anis Matoussi\textsuperscript{1} \and Wissal Sabbagh\textsuperscript{1}}

\begin{document}

\maketitle
\rightskip=2cm
\leftskip=2cm

\footnotetext[1]{\small Laboratoire Manceau de Math\'ematiques \& FR CNRS N\textsuperscript{o} 2962, Institut du Risque et de l'Assurance, Le Mans University}
\footnotetext[2]{Centre de Mathématiques Appliquées (CMAP), CNRS, École polytechnique, Institut Polytechnique de Paris, 91120 Palaiseau,
France. The author is funded by the European Union (ERC, SINGER, 101054787). Views and opinions
expressed are however those of the author only and do not necessarily
reflect those of the European Union or the European Research Council. Neither
the European Union nor the granting authority can be held responsible
for them.
}

\begin{small}
\begin{center}
\section*{Abstract}
\end{center}
In this paper, we present a probabilistic numerical method for a class of forward utilities in a stochastic factor model. For this purpose, we use the representation of forward utilities using the ergodic Backward Stochastic Differential Equations (eBSDEs) introduced by Liang and Zariphopoulou in \cite{liang2017representation}. We establish a connection between the solution of the ergodic BSDE and the solution of an associated BSDE with random terminal time $\tau$, defined as the hitting time of the positive recurrent stochastic factor. The viewpoint based on BSDEs with random horizon yields a new characterization of the ergodic cost $\lambda$ which is a part of the solution of the eBSDEs. In particular, for a certain class of eBSDEs with quadratic generator, the Cole-Hopf transformation leads to a semi-explicit representation of the solution as well as a new expression of the ergodic cost $\lambda$. The latter can be estimated with Monte Carlo methods. We also propose two new deep learning numerical schemes for eBSDEs. Finally, we present numerical results for different examples of eBSDEs and  forward utilities together with the associated investment strategies. 
\end{small}

\rightskip=0cm
\leftskip=0cm

\vspace{1cm}

\newpage
\section*{Introduction}

\noindent
We are interested in the numerical approximation of certain classes of forward performance/utility processes, and their associated optimal decision criterion. Introduced by \cite{musiela2006investments}, forward utilities offer an interesting alternative to the classical setting of expected utility maximization at a terminal time. This forward-looking approach enables the dynamic adjustment of the decision criteria, starting from preferences which are known at an initial time, rather than imposing a potentially distant and arbitrary time horizon. The preferences of an agent are, thus, described by a (random) dynamic utility ($U(t,\cdot))$. The decision criterion maintains time consistency within the given investment or decision-making context, in the sense that if $X_t^\pi$ is the process (typically the wealth) resulting from the admissible decision/strategy $\pi$, then the preference process $U(t, X_t^\pi)$ is a supermartingale, and there exists an optimal strategy such that the preference process is a martingale.

Since their introduction, there has been 
tremendous theoretical developments in the field of forward utilities. In a general setting, \cite{nicole2013exact} established a sufficient condition for time-consistency when the dynamic utility is an It\^o random field. The forward utility verifies a non linear SPDE of HJB type. This work has been extended to forward utility of investment and consumption in \cite{el2018consistent}, and has been applied, for instance, to derive forward utilities in stochastic factor market models (see e.g. \cite{nadtochiy2014class}, \cite{avanesyan2020construction}).
 Forward utilities have found diverse applications over recent years, including but not limited to option valuation, insurance, mean field games (\cite{dos2021forward}), long term interest rate modeling (\cite{el2022ramsey}), risk measures (\cite{chong2019pricing}) or more recently pension design (\cite{hillairet2022time}, \cite{ng2024optimal}). Surprisingly, the subject of numerical methods for forward utilities remains largely unexplored, despite its critical importance for practical applications. In \cite{gobet2018convergence}, a general approach is proposed using strong approximations of compounds of random maps. In this paper, we take a different approach to introduce new numerical schemes for the class of so-called homothetic forward utilities, taking advantage of the representation of these processes using ergodic BSDEs, introduced in \cite{liang2017representation}.

We investigate the representation of an agent's forward preferences investing in an incomplete financial market, where stock price dynamics are driven by a stochastic factor $(V_{t})_{t \geq 0}$. Homothetic forward utilities are expressed as separable functionals, denoted by $U(t,x) = u(x)e^{f(t,V_t)}$, where $u$ is a standard exponential or power utility function (the expression is additive in the logarithmic case). 
The main result of \cite{liang2017representation} provides a representation of the function $f$ with mean of the unique Markovian solution of a related ergodic BSDE.

Ergodic BSDEs have first been introduced in \cite{fuhrman2009ergodic}, with the aim to study an optimal ergodic control problem, expressed as the minimization of an averaged cost function over an infinite time horizon. 
Formally, the solution of an ergodic BSDE, which is an infinite horizon BSDE, is a triplet $(Y, Z, \lambda)$, where $Y$ and $Z$ are adapted processes and $\lambda$ is a real number, which solves
\begin{eqnarray}
\label{eBDSEintro}
Y_{t} = Y_{T} + \int_{t}^{T} (F(V_{s}, Z_{s}) - \lambda)ds - \int_{t}^{T} Z_{t}^{\top}dW_{t}, \quad \forall \, 0 \leq t \leq T < +\infty.
\end{eqnarray}
In \cite{fuhrman2009ergodic}, this equation is studied under Lipschitz assumptions on the driver, for a stochastic factor $V$ with constant volatility and drift that satisfy a dissipative condition. The dissipativity assumption has been relaxed in \cite{debussche2011ergodic} for stochastic factors with constant volatility in general Hilbert spaces and for non constant and possibly unbounded volatility in \cite{hu2019ergodic}. In the constant volatility framework, the component $Z$ of the ergodic BSDE solution is bounded, allowing \cite{liang2017representation} to obtain existence and uniqueness results of Markovian solutions $(y(V_t), z(V_t), \lambda)_{t\geq 0}$, for drivers $F$ which are only locally Lipschitz in $z$.\\

In this paper, we develop numerical schemes for approximation of Markovian solutions of the general class of ergodic BSDEs introduced in \cite{liang2017representation}, which includes the ergodic BSDEs used in the representation of homothetic utilities. There are two main challenges in the simulation of such equations:
\begin{enumerate}
    \item There is an additional real unknown $\lambda$. The usual backward discretization equation for $Y$ thus depends on $\lambda$.
    \item This is an infinite horizon BSDE, which has to hold for all $T > 0$, and for all $0 \leq t \leq T$. Thus, there is no terminal condition as in the simulation of finite horizon BSDE.
\end{enumerate}
The unkwown ergodic cost $\lambda$ can be interpreted in several ways. First, it is the long term growth rate of an associated risk sensitive control problem, as mentioned in \cite{liang2017representation}. It can also be represented as the linear growth rate of the initial value of the solution of an analogous finite horizon BSDE with respect to the terminal time $T$ when the latter goes to infinity, see \cite{hu2019ergodic}. However, numerical schemes for the simulation of BSDEs become unstable for large horizons, and as a result this representation cannot be used for the numerical approximation of the ergodic cost.

Over the past few years, machine learning algorithms have been extensively studied since they can be used to solve high dimensional non-linear PDEs, based on the BSDE representation of their solution (see, for example \cite{han2017deep}, \cite{chan2019machine}, \cite{hure2020deep}, \cite{germain2021neural}, \cite{kapllani2020deep}). Two main types of neural network algorithm have been developed. The first relies on a global loss function for solving BSDEs and was initially proposed in \cite{han2017deep}. The Deep BSDE solver consists in the training of as many neural networks as time steps to approximate the component $Z$ of the solution. The process $Y$ is computed with a forward discretization starting from $Y_{0}$. Then, $Y_{0}$ and the neural networks parameters are optimized according to a loss function on the terminal value of the discretized scheme. A convergence study of the Deep BSDE is developed in \cite{han2020convergence} and \cite{chan2019machine} shows that sharing one neural network across all time steps is more efficient. The second class of algorithms relies on a local approach and consists of solving local optimization problems at each time step. Firstly introduced in \cite{germain2021neural}, \cite{kapllani2020deep}, these methods use two neural networks to approximate both processes $Y$ and $Z$. Local loss functions are constructed based on the iteration of time discretization of BSDEs with the terminal condition.\\

Herein, we take advantage of the recurrence property of the stochastic factor $V$ in order to provide a horizon as well as a terminal condition to the problem of simulating the solution of an ergodic BSDE. In the case of ergodic BSDEs derived from forward utilities, an initial condition $Y_0$ is naturally given since the initial agent's utility $u_0$ is known. 
This allows to introduce a random horizon $\tau$ to the ergodic BSDE, defined as a return time of the one dimensional diffusion $V$ to $v_{0}$. Under the dissipativity assumption (Assumption \ref{weakdissass}), this stopping time is almost surely finite and $ Y_\tau = Y_0 $. 
Then the solution of the ergodic BSDE \eqref{eBDSEintro} is also a solution of an \lq ergodic\rq\,BSDE with random terminal time $\tau$ and fixed initial condition. Under additional integrability assumptions on $\tau$, we prove that it is the unique solution of the \lq ergodic\rq\,BSDE with random terminal time, using uniqueness result from \cite{pardoux1998backward} and the fact that the unknown $\lambda$ is uniquely determined by the fixed initial condition.

This new representation of ergodic BSDEs is particularly useful in order to design numerical schemes. When the driver $F$ is linear in $Z$, the representation result for linear BSDE leads to a semi-explicit representation of the solution as well as a new characterization of the ergodic cost $\lambda$, depending on the random horizon $\tau$. This can be applied to obtain an expression for the ergodic cost which is actually linked with exponential forward utilities when there is no constraint on the portfolio processes. Applying the Cole-Hopf transformation, this last result can be extended for purely quadratic drivers, as it is the case for power utility and with no constraints on the portfolio. For these examples, regression and Monte Carlo methods can be used to simulate the solution of the associated ergodic BSDEs.

For the approximation scheme, we introduce the Euler discretization of the stochastic factor $V$ and the associated estimation of the horizon time $\tau$. Results from \cite{geiss2017} related to the Euler estimation of the exit time of a diffusion from a smooth domain apply in our setting and provide a bound on the error in $L^{1}$ of the approximation error on $\tau$. We then present a backward scheme for ergodic BSDE using the correspondence with BSDE with random terminal time (see, for example, \cite{bouchard2009discrete}, \cite{bouchard2009strong}) and establish a bound for the related discrete-time approximation error in terms of the quantities $\Rr(Z)_{\Hr^{2}}^{\pi}$, $\ABS{\lambda - \bar{\lambda}}$ and $\e \SBRA{\ABS{\tau - \bar{\tau}}}$. Compared to \cite{bouchard2009strong}, the error bound does not depend on $\Rr(Y)_{\Sr^{2}}^{\pi}$ since the generator of ergodic BSDEs we consider is independent of $y$. However, we get an additional term depending on the estimation of the ergodic cost $\lambda$.

We also present two deep learning based methods for the simulation of ergodic BSDEs, which allow to tackle simultaneously the approximation of the ergodic cost $\lambda$ and the unknown processes $Y$ and $Z$. In the context of ergodic BSDEs, the initial value $Y_{0}$ is known already. We use a forward discretization starting from $Y_{0} = y_{0}$, and approximate $\lambda$ as a trainable parameter of the model. We first present a global solver denoted by GeBSDE, approximating $Z$ with one neural network common across all time steps. The optimization is performed according to a loss function at the random horizon $\tau$, the output aiming to match the terminal value $Y_{\tau} = Y_{0}$. We, then, present a second algorithm denoted by LAeBSDE, based on a local approach, approximating $Y$ and $Z$ with two distinct neural networks. The optimization is, then, performed according to the aggregation of local loss functions at each time step. We provide some numerical tests to evaluate the performance of both algorithms. We investigate two examples with explicit solutions, taken from \cite{hu2019ergodic}, and two examples adapted from \cite{liang2017representation}, with a driver representing power forward utilities. In the latter case, our algorithm also allows the simulation of the optimal strategy.

\vspace{0.3cm}
\noindent
The paper is organized as follows. In Section \ref{section:forwardutilities}, we briefly recall the stochastic factor model of \cite{liang2017representation}, as well as the class of homothetic forward utilities and their connection with ergodic BSDEs. In Section \ref{section:connection}, we establish some recurrence properties of the stochastic factor $V$, and then introduce the ergodic BSDE with random terminal time $\tau$, with fixed initial and terminal condition. We show that the solutions of such equation coincides with the Markovian solution of the ergodic BSDE. Using this representation, we first study in Section \ref{section:numerical} a backward discretization and the associated error estimate. Finally we present the deep learning algorithms for ergodic BSDE and the numerical results in Section \ref{section:simulation}.

\newpage
\noindent\textbf{Notations:}\\
All stochastic processes in the sequel are defined on a standard probability space $\PAR{\Omega, \f, \F, \p}$, where the filtration $\f = (\F_{t})_{t \geq 0}$ is the natural filtration generated by a $d$-dimensional Brownian motion $W$, and is assumed right continuous and complete. For $x \in \R^{d}$, we denote by $x^{\top}$ the transpose of vector $x$, $\NRM{.}$ the usual norm $\NRM{x} = \Tr(x x^{\top})^{\frac{1}{2}}$ and $\dist(x, \Pi)$ the distance function of $x$ to a closed convex subset $\Pi \subset \R^{d}$. We denote by $L^{2}$ the space of square integrable random variables and, also, introduce the usual space of solutions for $\gamma \in \R$ and $\tau$ being a $\f$ stopping time:
\begin{eqnarray*}
\Sr^{2}(\gamma, \tau) &=& \BRA{(\varphi_{t})_{t \geq 0}, \, \text{real valued progressively measurable process} \,\, \text{s.t.} \, \e \SBRA{\underset{0 \leq s \leq \tau}{\sup} e^{\gamma s} \ABS{\varphi_{s}}^{2}} < \infty.} \\
\end{eqnarray*}

\section{Forward utilities and ergodic BSDEs} \label{section:forwardutilities}

The aim herein is to investigate the numerical approximation of different classes of homothetic forward utilities, as introduced in \cite{liang2017representation}, which model an agent's dynamic preferences as she invests in a stochastic factor financial market. \\
We start by introducing the setting and related results of \cite{liang2017representation}. In particular, we are interested in the representation of homothetic forward utilities involving the unique Markovian solution of a certain class of ergodic BSDEs. This representation motivates our study of numerical approximations for ergodic BSDEs.

\subsection{Forward utilities and link with ergodic BSDEs}

Dynamic utilities generalize the notion of the classical utility function. Formally, a dynamic utility $U = (t,x,\omega) \in \R^+ \times \R^+ \times \Omega \to \R$ is a collection of random utility functions such that:
\begin{itemize}[-]
\item For all $t \geq 0$, and for all $x \in \R^+$, $U (t, x)$ is $\mathcal F_t$-measurable.
\item The functions $x \in \R^+ \mapsto U(t,x,\omega)$ are nonnegative, strictly concave increasing functions of class $\mathcal C^{2}$ on $]0, \infty[$, $(\omega,t)$ a.s.
\item $u_0:= U(0, \cdot)$ is a standard (deterministic) utility function.
\end{itemize}

Homothetic utilities, introduced in \cite{liang2017representation}, are functions of the agent's wealth $x$ and are characterized by one of the following forward utilities:
\begin{flalign}
&\textit{- Logarithmic case}: \quad U(t, x) = ln(x) + f(V_{t}, t), \label{logut} \\
&\textit{- Exponential case}: \quad U(t, x) = - e^{- \gamma x + f(V_{t}, t)}, \, \gamma \in (0, 1), \label{exput} \\
&\textit{- Power case}: \quad U(t, x) = \frac{x^{\delta}}{\delta} e^{f(V_{t}, t)}, \, \delta\in(0,1), \, \label{powut} 
\end{flalign}
where $f$ is a deterministic function to be specified hereafter, and $V$ is a $d'$-dimensional diffusion process with local characteristics $\mu: \R^{d'} \to \R^{d'}$ and constant volatility matrix $\kappa$, solving
\begin{eqnarray} \label{stochfact}
dV_{t}^{i} = \mu^{i}(V_{t})dt + \sum_{j=1}^{d} \kappa^{ij}dW_{t}^{j}, \quad V_{0}^{i} \in \R.
\end{eqnarray}

\noindent
The agent invests in an incomplete market consisting of one riskless bond and $n$ stocks. Assuming the bond to be the numeraire, the stock price dynamics, discounted by the interest rate, are given for $i=1, ..., n,$ by
\begin{eqnarray}
dS_{t}^{i} = S_{t}^{i} \PAR{b^{i}(V_{t})dt + \sum_{j=1}^{d}\sigma^{ij}(V_{t})dW_{t}^{j}},
\end{eqnarray}
where $V$ is a $d'$-dimensional stochastic factor given by \eqref{stochfact}, and satisfy the following assumption.
\begin{ass} \label{asslispch}
\begin{enumerate}
            \item The functions $b=(b^i)_{1\leq i \leq n}$ and $\sigma=(\sigma^{ij})_{1\leq i \leq n\atop 1 \leq j \leq d}$ are uniformly bounded and for all $v \in \R ^{d'}$, the matrix $\sigma(v)$ has full row rank $n$.
            \item The function $\theta= \sigma^{\top} (\sigma \sigma^{\top})^{-1} b$ is a uniformly bounded and Lipschitz continuous function. 
            The risk premium vector is noted by $\theta(V_t)$ taking values in $\R^n$.                    
\end{enumerate}
\end{ass}

The agent invests a proportion $\bar{\pi} = \PAR{\bar{\pi}^{1}, ..., \bar{\pi}^{n}}^{\top}$ of her wealth $X^{\pi}$ in the $n$ risky assets. For an initial value $X_{0}^{\pi} = x_{0} \in \R^+$, assuming the self-financing condition holds and rescaling the strategy vector by the volatility, the wealth process $X$ evolves as 
\begin{eqnarray} \label{wealth}
dX_{t}^{\pi} = X_{t}^{\pi} \pi_{t} \cdot \PAR{\theta(V_{t})dt + dW_{t}}, \quad \pi_t = \sigma(V_t)^{\top} \bar \pi_t \in \R^d . 
\end{eqnarray}
For each $t\geq 0$, the strategy $(\pi_t)_{t \geq 0}$ is assumed to be in a closed and convex set $\Pi  \subset \R^{d}$. Admissible strategies are, also, required to be BMO. We refer to \cite{liang2017representation} for further details. 

\begin{rmq}
For exponential performance process, it is more convenient to use the discounted amount of wealth invested in the stock $\alpha_{t}  = X_t^\pi  \pi_t $ as control variable, leading to the following wealth process dynamics
\begin{eqnarray}
dX_{t}^{\alpha} = \alpha_{t}^{\top} \PAR{\theta(V_{t})dt + dW_{t}}.
\end{eqnarray}
\end{rmq}

\paragraph*{}  A dynamic utility is called a \textit{forward utility} if the consistency property holds: the utility is a supermartingale along the wealth process for any admissible control $\pi$ and a martingale along the optimal wealth process. The optimal strategy thus gives maximal satisfaction to the agent, which is preserved at all times in the future. This additional time consistency property makes the notion of forward utilities coherent with the dynamic programming principle.

\begin{defi}[\textit{Forward utility}] \label{progutdef}
 A forward utility is a dynamic utility $U$ satisfying the time consistency property:
\begin{itemize}[-]
    \item  For any admissible strategy $\pi$, $U(t, X_{t}^{\pi})$, $t\geq 0$, is a supermartingale. 
    \item There exists an admissible strategy $\pi^{*}$ such that $U(t, X_{t}^{\pi^{*}})$, $t\geq 0$, is a martingale.
\end{itemize}
\end{defi}

When $U$ is an It\^o-random field with sufficient regularity conditions on its local characteristics, a sufficient consistency condition of HJB type, characterizing the drift of forward utilities, as well as the optimal strategy under this condition are obtained in \cite{MZ2010}, \cite{nicole2013exact} and \cite{SSZ2016}. In particular, $U$ is solution of a non-linear HJB-SPDE under this sufficient assumption for consistency. In the case of the homothetic forward utilities \eqref{logut}-\eqref{powut} with stochastic factor, the HJB-SPDE on $U$ is equivalent to a PDE for the deterministic function $f$ (see \cite{liang2017representation}). As pointed out by the authors, the problem is ill posed. However, a characterisation of $f$ via the Markovian solution of a related ergodic BSDE is given in \cite{liang2017representation}, which allows us to develop numerical schemes for homothetic forward utilities.

\paragraph{Homothetic forward utility and ergodic BSDE} Formally, an ergodic BSDE with generator $F$ is a backward stochastic differential equation with infinite horizon, whose solution is a triplet $(Y, Z, \lambda)$ where $Y, Z$ are adapted processes and $\lambda \in \R$, and satisfies for any $T>0$, $\p$-a.s for any $0 \leq t \leq T$
\begin{eqnarray} \label{ebsde}
Y_{t} &=& Y_{T} + \int_{t}^{T} F(V_{s}, Z_{s})ds - \lambda (T-t) - \int_{t}^{T} Z_{s}^{\top} dW_{s} \label{introebsde} \\
dV_{t}^{i} &=& \mu^{i}(V_{t})dt + \sum_{j=1}^{d} \kappa^{ij}dW_{t}^{j}, \quad V_{0}^{i} \in \R,\quad 1\leq i\leq d^{'}.  
\end{eqnarray}
This class of ergodic BDSE was first introduced in \cite{fuhrman2009ergodic}. Existence and uniqueness results in our framework are recalled in Section \ref{section:eBSDE} below. Note that the infinite horizon is coherent with the willingness to adapt dynamically the utility for all upcoming times.

\paragraph{} Next, we introduce the generators associated with the various homothetic forward utilities defined in \eqref{logut}, \eqref{exput}, \eqref{powut}.
\begin{itemize}[-]
\item \textit{Logarithmic case}: for $v \in \R^{d'}$
\begin{eqnarray} \label{driver_log}
F_{\log}(v) = - \frac{1}{2} \dist^{2} \PAR{\Pi, \theta(v)} + \frac{1}{2} \NRM{\theta(v)}^{2}.
\end{eqnarray}
\item \textit{Exponential case}: for $(v, z) \in \R^{d'} \times \R^{d}$
\begin{eqnarray} \label{driver_exp}
F_{\exp}(v, z) = \frac{1}{2} \gamma^{2} \dist^{2} \PAR{\Pi, \frac{z + \theta(v)}{\gamma}} - \frac{1}{2} \NRM{z + \theta(v)}^{2} + \frac{1}{2} \NRM{z}^{2}.
\end{eqnarray}
\item \textit{Power case}: for $(v, z) \in \R^{d'} \times \R^{d}$
\begin{eqnarray} \label{driver_zar}
   F^{\delta}(v, z) = \frac{\delta(\delta-1)}{2} \dist^{2}\PAR{\Pi, \frac{\theta(v) + z}{1- \delta}} + \frac{\delta}{2(1 - \delta)} \NRM{\theta(v) + z}^{2} + \frac{1}{2}\NRM{z}^{2},
\end{eqnarray}
\end{itemize}

If equation \eqref{ebsde} with one of the above generators admits a Markovian solution \newline $(y(V_{t}), z(V_{t}), \lambda)_{t\geq 0}$, one can show that the homothetic forward utilities $U$ defined by \eqref{logut}-\eqref{powut}, with 
\begin{equation*}
    f(v, t) = y(v) - \lambda t, 
\end{equation*} 
are time forward homothetic utilities, as defined in Definition \ref{progutdef}. Next, we summarize these results, while existence and uniqueness results of Markovian solutions to \eqref{ebsde} are recalled the next section.
\begin{prop}[Theorem $3.2$ and $4.2$, \cite{liang2017representation}]
Let $(y(V_{t}), z(V_{t}), \lambda)_{t \geq 0}$ be a Markovian solution of the ergodic BSDE \eqref{ebsde} with driver $F$ given by \eqref{driver_log} (resp. \eqref{driver_exp}, \eqref{driver_zar}). \\
Then, the associated logarithmic (resp. exponential, power) utility $U$ \eqref{logut} (resp. \eqref{exput}, \eqref{powut}) with $f(., t) = y(.) - \lambda t$ is a forward utility.
Furthermore, the optimal strategy is given by:
\begin{align}
     \label{pi_log}
 \text{Logarithmic case} & \quad   \pi_t^* = \text{Proj}_{\Pi}\big(\theta(V_t)\big). &\\
\label{pi_exp}
 \text{Exponential  case} &  \quad \alpha_t^* = \text{Proj}_{\Pi}\big(\frac{z(V_t)+ \theta(V_t)}{\gamma}\big).& \\
\label{pi_pow}
 \text{Power case}  & \quad  \pi_t^* = \text{Proj}_{\Pi}\big(\frac{z(V_t)+ \theta(V_t)}{1-\delta}\big). & 
 \end{align}
\end{prop}

Motivated by this representation, the aim of this paper is to propose numerical methods for the simulation of Markovian solutions of ergodic BSDEs that allow us to approximate utilities (\eqref{logut}-\eqref{powut}) and their optimal strategies. As a matter of fact, we study a larger class of ergodic BSDE, introduced below.

\subsection{Markovian solution of ergodic BSDEs} \label{section:eBSDE}

Ergodic BSDEs have first been studied in \cite{fuhrman2009ergodic} under a dissipativity assumption on the stochastic factor $V$ to solve an ergodic stochastic control problem. The assumptions on the stochastic factor have been relaxed in \cite{debussche2011ergodic} with a weak dissipative condition and in \cite{hu2019ergodic} for non constant and possibly unbounded volatility. Ergodic BSDEs are usually studied under Lipschitz condition on the generator $F$. When the stochastic factor's volatility is constant, the component $Z$ of the solution to the eBSDE is bounded, which allows the driver to only be locally Lipschitz in $z$. We will work within the framework of \cite{liang2017representation} with a stochastic factor satisfying a strong dissipativity assumption and constant volatility. This framework leads to the existence of a Markovian solution to the ergodic BSDE \eqref{ebsde} such that $Z$ is bounded, and thus allows the generator to have quadratic growth in $z$. 

\begin{ass} \label{weakdissass}
There exists a constant $C_{\mu} > 0$ such that for any $v, \bar{v} \in \, \R^{d}$
            \begin{eqnarray} \label{disscond}
            \PAR{\mu(v) - \mu(\bar{v})}^{\top}(v - \bar{v}) \leq - C_{\mu}\NRM{v - \bar{v}}^{2}.
            \end{eqnarray}
The volatility matrix $\kappa = \PAR{\kappa_{ij}}_{\underset{1 \leq j \leq d}{1 \leq i \leq d'}}$ is such that $\kappa \kappa^{\top}$ is positive definite.
\end{ass}


Under Assumption \ref{weakdissass}, from a direct application of Gronwall's lemma, the diffusion $V$ is exponentially ergodic. The authors in \cite{hu2019ergodic} generalized this result under a weak dissipative assumption. This property is essential for the correspondence with random time horizon BSDE and the algorithm we present in the sequel.

\begin{ass} \label{Fgrowthass}
The following properties hold:
\begin{itemize}
\item [i)]There exists a positive constant $K$ such that $\forall v \in \R^{d'}$, $\ABS{F(v, 0)} \leq K$.
\item [ii)]There exists positive constants $C_{v}$ and $C_{z}$ such that $\forall v, \bar{v} \, \in \R^{d'}$, $\forall z, \bar{z} \, \in \R^{d}$
\begin{eqnarray}
\ABS{F(v, z) - F(\bar{v}, z)} &\leq& C_{v} \PAR{1 + \NRM{z}} \NRM{v-\bar{v}}, \label{Ftroncv} \\
\ABS{F(v, z) - F(v, \bar{z})} &\leq& C_{z} \PAR{1 + \NRM{z} + \NRM{\bar{z}}} \NRM{z - \bar{z}}. \label{Ftroncz}
\end{eqnarray}
Moreover, we require that $C_{v} < C_{\mu}$ with $C_{\mu}$ in \eqref{disscond}.
\end{itemize}
\end{ass}

Note that under Assumption \ref{asslispch}, the function $\theta$ is bounded and Lipschitz, which yields that every generator \eqref{driver_log}, \eqref{driver_exp} and \eqref{driver_zar} introduced in the previous section satisfies Assumption \ref{Fgrowthass}. We recall the existence result for eBSDE studied in \cite{liang2017representation}.

\begin{prop}[Existence - \cite{liang2017representation}] \label{existebsde}
Under Assumptions \ref{weakdissass} and \ref{Fgrowthass}, the ergodic BSDE \eqref{ebsde} admits a Markovian solution $(y(V_t), z(V_t), \lambda)_{t\geq 0}$ such that $y(.)$ is sub-linear and $z(.)$ is bounded by $Z_{\max} = \NRM{\kappa} \displaystyle\frac{C_{v}}{C_{\mu} - C_{v}}$.
\end{prop}

Assumptions \ref{weakdissass} and \ref{Fgrowthass} thus provide the existence of a Markovian solution $(y(V_t), z(V_t), \lambda)_{t\geq 0}$ to \eqref{ebsde} where $z(.)$ is bounded, which is particularly convenient to apply our work to ergodic BSDE with quadratic driver $F$ as for example to simulate exponential and power forward utilities. In fact, working with the truncated driver $F \circ \varphi_{Z_{max}}$ where $\varphi_{Z_{max}}$ is the projection on the centered ball of $\R^{d}$ of radius $Z_{max}$, the application $F \circ \varphi_{Z_{max}}$ is then Lispchitz in $v$ and $z$, namely:
\begin{eqnarray}
\ABS{F \circ \varphi_{Z_{\max}} (v, z) - F \circ \varphi_{Z_{\max}}(\bar{v}, z)} &\leq& C_{v} \PAR{1 + Z_{\max}} \NRM{v - \bar{v}}, \label{Flipschv} \\
\ABS{F \circ \varphi_{Z_{\max}} (v, z) - F \circ \varphi_{Z_{\max}} (v, \bar{z})} &\leq& C_{z} (1 + 2 Z_{\max}) \NRM{z - \bar{z}}. \label{Flipschz}
\end{eqnarray}
The uniqueness of the Markovian solution to \eqref{ebsde} is usually stated up to a constant, by fixing one point of the solution, typically $y(0) = 0$. The proof follows the arguments from \cite{debussche2011ergodic} and \cite{fuhrman2009ergodic} when the driver is Lipschitz.

\begin{thm}\label{Uthmebsde}[Uniqueness - \cite{liang2017representation}]
Assume that Assumptions \ref{weakdissass} and \ref{Fgrowthass} hold. Let $(y, z)$, $(\tilde{y}, \tilde{z})$, two functions such that:
\begin{itemize}
    \item [i)]$y$, $\tilde{y}: \R^{d} \to \R$ are continuous, sub-linear and $y(0) = \tilde{y} (0)$. 
    \item [ii)]$z$, $\tilde{z}: \R^{d} \to (\R^{*})^d$ are measurable and bounded by $Z_{max}$.
\end{itemize}
Also assume that for some constants $\lambda, \tilde{\lambda}$ and for all $v \in \R^{d'}$, the triplets $\PAR{y(V_{t}^{v}), z(V_{t}^{v}), \lambda}_{t\geq 0}$ and $\PAR{\tilde{y}(V_{t}^{v}), \tilde{z}(V_{t}^{v}), \tilde{\lambda}}_{t\geq 0}$ satisfy the ergodic BSDE \eqref{ebsde}.

Then $\lambda = \tilde{\lambda}$, $y(V_t^{v}) = \tilde{y}(V_t^{v})$ and $z(V_{t}^{v}) = \tilde{z}(V_{t}^{v})$ $\p$-a.s and for a.e $t \geq 0$. 
\end{thm}

 \paragraph{Initial condition} In the case of ergodic BSDEs derived from forward utilities, an initial condition $Y_0 =y_0$ for the process $Y$ is naturally fixed since the initial agent's utility and wealth $u_0 (x_0)$ are known. 
For instance in the case of power forward utilities \eqref{powut}, we have  
\begin{equation*}
y_0=  \log(\delta u_0(x_0))  - \delta\log(x_0). 
\end{equation*}
We are thus interested in solutions of the ergodic BSDE \eqref{ebsde} with \textit{fixed} initial condition
\begin{eqnarray}
\nonumber Y_{t} &=& Y_{T} + \int_{t}^{T} F(V_{s}, Z_{s})ds - \lambda (T-t) - \int_{t}^{T} Z_{s}^{\top} dW_{s}, \quad \forall \; 0\leq t \leq T.  \\
Y_0 & = &  y_0.  \label{ebsdeInitCond}
\end{eqnarray}
Using the notation of Theorem \ref{Uthmebsde}, let $(y(V_t), z(V_t), \lambda)_{t\geq 0}$ be the unique solution of \eqref{ebsde}, such that $y(0) = 0$. Then, $(Y,  Z,\lambda)$, with 
\begin{equation}
\label{eq:ebsdeInitcondYZ}
    Y_t = y(V_t) + y_0 - y(V_0), \quad Z_t = z(V_t) \quad \forall \;  t\geq 0, 
\end{equation}
is a solution of \eqref{ebsdeInitCond}. Note that the solution is not anymore Markovian, since $Y$ depends on the stochastic factor's initial condition $V_0$. Next, we define the unique solution of the ergodic BSDE \eqref{ebsdeInitCond} with fixed initial condition as the triplet $(Y, Z, \lambda)_{t\geq 0}$, with $(Y,Z)$ as in \eqref{eq:ebsdeInitcondYZ}. With a sligh abuse of language, we will occasionally refer to this solution as the unique "Markovian" solution of \eqref{ebsdeInitCond}.

\section{Connection with BSDE with random terminal time} \label{section:connection}

Simulating the ergodic BSDE \eqref{ebsdeInitCond} presents several additional challenges, to the simulation of standard BSDEs with finite time horizons. In particular:
\begin{enumerate}
\item There is an additional unknown $\lambda \in \R$, which controls the time growth of the component $Y$.
\item It is an infinite horizon backward stochastic differential equation, and thus the equality
\begin{eqnarray*}
Y_{t} = Y_{T} + \int_{t}^{T} \PAR{F(V_{s}, Z_{s}) - \lambda}ds - \int_{t}^{T} Z_{s}^{\top} dW_{s},
\end{eqnarray*}
needs to hold for \it{any} $ T > 0$, \, for all $0 \leq t \leq T.$ Thus, there is no known terminal condition as is usually the case in numerical schemes for BSDEs with backward time discretization.
\end{enumerate}
However, we may take advantage of recurrence properties of the stochastic factor $V$ to establish a connection between the ergodic BSDE and a BSDE with random terminal time, thus introducing a terminal condition. More precisely, when the stochastic factor $V$ is one dimensional ($d'=1$), knowing the initial condition $Y_0 $ of the ergodic BSDE allows us to introduce an analogous ``ergodic BSDE with random time horizon'' $\tau$. The infinite time horizon is replaced by $\tau$, the first return time after a minimal horizon $T$ of the diffusion $V$ to the initial point $V_{0}$. With this choice of random terminal time, the associated BSDE with random horizon has a known terminal condition $Y_\tau = Y_0$, and hence can be approximated numerically.

\subsection{First return time of the stochastic factor $V$} \label{recsect}

Unless stated otherwise, we assume  in the sequel that the stochastic factor $V$ is one dimensional ($d'=1$). We then fix a minimal horizon $T$ and start with some properties of the first return time after $T$ of the diffusion $V$ to its initial value, denoted by $v_{0}\in \R$,
\begin{eqnarray} \label{deftau}
\tau= \inf \BRA{t > T, \, V_{t} = v_{0}}.
\end{eqnarray}
The Lemma \ref{lem:expint} gives sufficient conditions for the exponential integrability of the hitting time $\tau$, in order to study the ``ergodic BSDE with random horizon'' introduced in the next section. Theorem 1.1 in \cite{loukianov2011spectral} provides sufficient conditions for exponential integrability of hitting time for continuous Markov processes. More precisely, the authors obtain a lower bound for the highest order of exponential moment of $\tau$ in terms of the scale function and the speed measure, uniformly with respect to the initial condition. We adapt this result to our framework:

\begin{lem}
\label{lem:expint}
Let $\displaystyle s(x) = \exp \PAR{- 2 \int_{0}^{x} \frac{\mu(u)}{\NRM{\kappa}^{2}} du}$
\begin{eqnarray*}
& &B_{v_{0}}^{+} := \underset{x \geq v_{0}}{\sup} \PAR{\int_{v_{0}}^{x} s(u)du \int_{x}^{+ \infty} \frac{2}{\NRM{\kappa}^{2} s(u)}du }, \\
\text{and} & &B_{v_{0}}^{-} := \underset{x \leq v_{0}}{\sup} \PAR{\int_{x}^{v_{0}} s(u)du \int_{- \infty}^{x} \frac{2}{\NRM{\kappa}^{2} s(u)}du }.
\end{eqnarray*}
Let $K_{z}:= C_{z} (1 + 2Z_{\max})$ be the Lipschitz constant with respect to $z$ of the truncated driver $F \circ \phi_{Z_{\max}}$ given in \eqref{Flipschz} and assume that 
\begin{equation}\label{lemma:Kz}
    K_{z}^{2} < \ \frac{1}{4 \max \PAR{B_{v_{0}}^{-}, B_{v_{0}}^{+}}}.
\end{equation}
Then, $\exists  \; \gamma > K_z^2$ such that 
\begin{equation}
    \mathbb E[\exp(\gamma \tau) ] <\infty. 
\end{equation}
\end{lem}
\begin{proof}
Let $\lambda \leq \displaystyle\frac{1}{4 \max \PAR{B_{v_{0}}^{-}, B_{v_{0}}^{+}}}$, and $\mu$ the unique invariant measure of $V$. Recall that $\tau \geq T$, and let
$   W(v)=   \e[\exp(\lambda \tau)| V_{T}=v]$. By the Markov property of V, 
\begin{equation*}
    \mathbb E[\exp(\lambda \tau) ] =  \mathbb E[W(V_{T})].
\end{equation*}
First, by Theorem 1.1 in \cite{loukianov2011spectral}, we have $W(v) < \infty,  \forall \; v \in \R, $
and equivalently (see Proposition 1.2 in \cite{loukianov2011spectral}), $$\displaystyle \mu(W)=\int_{-\infty}^{+\infty} W(v) \mu(d v) <\infty.$$
It remains to prove that $\e[W(V_{T})]<\infty$. Since the diffusion coefficient of the stochastic factor $V$ is constant, $V$ is a uniformly elliptic diffusion, and it follows that $W$ is Lyapunov function for the diffusion $V$ (see e.g. Theorem 2.3 in \cite{cattiaux2013}). \\
Using similar arguments as in the proof of Theorem 7 in \cite{hu2019ergodic}, the hypothesis of Theorem A.2 in \cite{weinan2001ergodicity} is verified with $V= W$. Applying the result to $f=W$, we obtain that 
\begin{align*}
 \e[W(V_{T})] &\leq    |\e[W(V_{T})] - \mu(W) | + \mu(W) \\
 &\leq 2Be^{-rT} + \mu(W)  <\infty.  
\end{align*}

\end{proof}

\begin{rmq}
In higher dimensions, results on positive recurrence of Markov diffusion process associated with a stochastic differential equation stands for hitting time of any non empty open set in $\R^{d}$. Working with $\tau^{\epsilon}$ the first hitting time of the Euclidean ball of $\R^{d}$ of center $v_{0}$ and radius $\epsilon$, we may only expect the value $y(V_{\tau^{\epsilon}})$ to be close to $y(v_0)$, by the continuity of the Markovian solution $y$. We can thus only hope to obtain an approximation of the solution of the ergodic BSDE, by investigating stability results for ergodic BSDE with perturbed terminal condition. This is left for future work.
\end{rmq}

\subsection{BSDE with random terminal time}

Under Assumptions \ref{weakdissass} and \ref{Fgrowthass}, there exists a unique solution $(y(V_t) + y_0 - y(V_0), z(V_t), \lambda)_{t\geq 0 }$ to the ergodic BSDE \eqref{ebsdeInitCond} such that $Y_0 =y_0$, $y$ is sub-linear with respect to $v$, and $z$ is bounded by $Z_{max}$. By construction, $(Y_t, Z_t, \lambda)_{t\geq 0 }$ is also solution of the following ``ergodic'' BSDE with random time horizon and fixed initial condition
\begin{eqnarray} \label{rdtBSDE}
Y^{r}_{t} &=& Y^{r}_{\tau} + \int_{t}^{\tau} F(V_{s}, Z^{r}_{s})ds - \lambda (\tau- t) - \int_{t}^{\tau} Z^{r \top}_{s} dW_{s}, \\
Y_{\tau}^{r} &=& Y_{0}^r = y_{0}, \nonumber
\end{eqnarray}
with $\tau$ the return time defined in \eqref{deftau}. 

In analogy, solution $(Y^{r}, Z^{r}, \lambda)$ of \eqref{rdtBSDE} can be studied directly. Theorem \ref{thm:BSDERandomtime} provides sufficient conditions, under which the solution of \eqref{rdtBSDE} coincides with the solution of the ergodic BSDE \eqref{ebsdeInitCond}. The setting is slightly different from usual BSDEs with random time horizon, since we have an additional unknown $\lambda \in \R$ and both the terminal and initial conditions are fixed. However, we can show that this constant is uniquely determined by the fixed initial and terminal conditions $ Y^r_\tau = Y^r_0=y_0$. Once $\lambda$ is known, the uniqueness of the solution to this BSDE with generator $ F(v, z) - \lambda$ can be obtained as a consequence of standard results for BSDEs with random terminal time. We apply here the general result Theorem $3.2$ from \cite{pardoux1998backward}, which requires integrability conditions for the stopping time $\tau$.  

\begin{thm}
\label{thm:BSDERandomtime}
Assume that \eqref{lemma:Kz} and Assumptions \ref{weakdissass}, \ref{Fgrowthass} hold.
Then, the ergodic BSDE with random time horizon \eqref{rdtBSDE} and fixed initial condition admits a unique solution $(Y, Z, \lambda)$, such that $ Y \in \Sr^{2}(\gamma, \tau)$ for all $\displaystyle K_{z}^{2} < \gamma \leq \frac{1}{4 \max(B_{v_{0}}^{-}, B_{v_{0}}^{+})}$ and $Z$ is bounded. 

\noindent In particular, $(Y, Z, \lambda)$ coincides on $[0,\tau]$ with the unique solution $(y(V_t) + y_0 - y(V_0), z(V_t), \lambda)_{t\geq 0 }$ of the ergodic BSDE \eqref{ebsdeInitCond} such that $Y_0 =y_0$, $y$ is sub-linear, and $z$ is bounded.
\end{thm}

\begin{preuve}
The proof is done in three steps. First we show that the existence of such a solution to \eqref{rdtBSDE} is obtained directly from the existence of a solution to the eBSDE \eqref{ebsdeInitCond}. Secondly, we establish the uniqueness of the parameter $\lambda$, using a linearisation technique. The uniqueness of $(Y,Z)$ is then obtained by applying Theorem $3.2$ from \cite{pardoux1998backward}.

\noindent
\textbf{Existence -} By construction, the unique solution $(y(V_{t})- y_0 + y(V_0), z(V_{t}), \lambda)_{t \geq 0}$ of the ergodic BSDE \eqref{ebsdeInitCond} is also solution of the BSDE with random terminal time \eqref{rdtBSDE}. By construction, $Z$ is bounded by $Z_{max}$ and thanks to Theorem 3.2 in \cite{pardoux1998backward}, we deduce that $ Y \in \Sr^{2}(\gamma, \tau)$ .
 
\noindent
\textbf{Uniqueness of $\lambda$ -} Let $(Y^{r}, Z^{r}, \lambda)$ and $(\bar{Y}^{r}, \bar{Z}^{r}, \bar{\lambda})$ be two solutions of the ergodic BSDE with random terminal time \eqref{rdtBSDE} such that $Z^{r}$ and $\bar{Z}^{r}$ are bounded. Denote $\Delta Y_{t} = Y_{t}^{r} - \bar{Y}_{t}^{r}, \, \Delta Z_{t} = Z_{t}^{r} - \bar{Z}_{t}^{r}$ and $\Delta \lambda = \lambda - \bar{\lambda}$. The initial and terminal values of those two solution being equal to $y_{0}$ yields that $\Delta Y_{0}^{r} = \Delta Y_{\tau}^{r} = 0.$\\

\noindent Let $T > 0$, the difference between the two equations, between $0$ and $T \wedge \tau$, gives
    \begin{eqnarray}
(T \wedge \tau) \Delta \lambda &=& \Delta Y_{T \wedge \tau} + \int_{0}^{T \wedge \tau} \Delta Z_{s}^\top \PAR{\gamma_{s}ds - dW_{s}}, \label{linearizationeq} \\
\text{where} \quad \gamma_{s} &=& \left\{
    \begin{array}{ll}
        \displaystyle\frac{F(V_{s}, Z_{s}^{r}) - F(V_{s}, \overline{Z_{s}^{r}})}{\ABS{\overline{Z_{s}^{r}} - Z_{s}^{r}}^{2}} \PAR{\overline{Z_{s}^{r}} - Z_{s}^{r}} \quad \text{if} \, \ABS{\overline{Z_{s}^{r}} - Z_{s}^{r}} \neq 0 \\
        0 \quad \text{otherwise}.
    \end{array}
\right.
\end{eqnarray}
From Assumption \ref{Fgrowthass} and the boundedness of $Z^{r}$ and $\bar{Z^{r}}$, the process $\gamma$ is bounded. Then, from Girsanov's theorem, there exists a probability measure $\tilde \p$ under which the process $\tilde{W_{t}} = - \int_{0}^{t} \gamma_{s}ds + W_{t}$, is a Brownian motion. The stopped process $M_{t \wedge T} := \int_{0}^{t \wedge T} \Delta Z_{s}^\top  d\tilde{W_{s}}$ is a martingale under $\tilde \p$. Taking expectation of \eqref{linearizationeq} yields
\begin{eqnarray*}
\tilde{\e} \SBRA{(T \wedge \tau) \Delta \lambda} = \tilde{\e}\SBRA{\Delta Y_{T \wedge \tau}} = \tilde{\e} \SBRA{\Delta Y_{\tau} \ind_{\tau \leq T}} + \tilde{\e} \SBRA{\Delta Y_{T} \ind_{\tau > T}}.
\end{eqnarray*}
The first expectation on the right is zero since $\Delta Y_{\tau} = 0$, by the definition of the stopping time $\tau$. For the second term, using the sub-linearity property of $Y$ and $\bar{Y}$, we obtain
\begin{eqnarray}
\tilde{\e} \SBRA{\Delta Y_{T} \ind_{\tau > T}} &\leq& \tilde{\e} \SBRA{\Delta Y_{T}^{2}} \tilde{\p}(\tau > T)  \nonumber \\ 
&\leq& C (1 + \tilde{\e} \SBRA{\ABS{V_{T}}^{2}}) \tilde{\p}(\tau > T). \label{lim_eq_T}
\end{eqnarray}
By Proposition 5 in \cite{hu2019ergodic}, $\underset{T \geq 0}{\sup} \tilde{\e} \SBRA{\ABS{V_{T}}^{2}} < \infty.$ Moreover, $\tau$ is almost surely finite under $\p$ so that is also almost surely finite under the equivalent probability measure $\tilde{\p}$. Then, taking the limit of \eqref{lim_eq_T} as $T$ goes to infinity implies that $\Delta \lambda = 0.$ Hence, the component $\lambda$ of the solution of \eqref{rdtBSDE} is necessarily equal to the $\lambda$ solution of the ergodic BSDE \eqref{ebsdeInitCond}.

\noindent
\textbf{Uniqueness -} Consider the BSDE with random terminal time $\tau$, terminal condition $Y^{r}_{\tau} = y_{0}$ and generator $F \circ \phi_{Z_{\max}} (v, z) - \lambda$, where $\lambda$ is fixed. The uniqueness can be obtained by applying Theorem $3.1$ in \cite{pardoux1998backward}.

\end{preuve}

The solution of the ergodic BSDE \eqref{ebsdeInitCond} thus coincides with the solution $(Y^{r}, Z^{r}, \lambda)$ of the ergodic BSDE with random time horizon and fixed initial condition \eqref{rdtBSDE} on $\SBRA{0, \tau}$. We will omit the subscript $r$ in the sequel. 
This idea is fundamental to get a numerical approximation and simulation of the ergodic BSDE \eqref{ebsdeInitCond}.
This allows us to adapt numerical schemes for the simulation of BSDEs with random time horizon, (see \cite{bouchard2009discrete} and \cite{bouchard2009strong}) to our framework in Section \ref{section:numerical}. Furthermore, using this representation of the ergodic BSDE, a new representation of the ergodic cost $\lambda$ can be obtained under some additional assumption of the driver $F$.

\subsection{Characterization of the ergodic cost 
$\lambda$ for a class of eBSDEs} \label{section:semiexplicit}

The ergodic cost $\lambda$ can be interpreted in several ways. As mentioned in \cite{liang2017representation}, it is the long term growth rate of a risk sensitive control problem. It can also be estimated by $\dfrac{Y_0^T}{T}$ when $T\to \infty$, with $Y_0^T$ being the initial value of the solution of a BSDE with finite horizon $T$ (Theorem 21 in \cite{hu2019ergodic}). However, numerical schemes for the simulation of BSDEs are known to be unstable for large horizon, which can lead to a significant error when using this approach to approximate $\lambda$. 

Our viewpoint of ergodic BSDE up to a random horizon, with fixed initial and terminal values offers a new characterization of the ergodic cost $\lambda$. In fact, considering $\lambda$ as a parameter in the generator of a classical BSDE with random terminal time, it can be understood as the solution of an optimization problem on the initial value $y_{0}$. This is also the idea behind the deep learning algorithms we present in Section \ref{section:simulation}. The adaptation of Proposition $1.3$ from \cite{el2008backward} for linear BSDEs leads to a semi-explicit expression of the solution of \eqref{rdtBSDE} as well as an expression of the ergodic cost $\lambda$ as the ratio of two expectations, that can be computed numerically. We investigate separately the case of linear ergodic BSDEs, and the one of ergodic BSDEs with purely quadratic generator which first requires to apply the Cole-Hopf transform.

\paragraph{}We first consider a random time horizon ergodic BSDE as \eqref{rdtBSDE} with driver $F$ only depending on the stochastic factor $V$. This is the framework for the representation of logarithmic forward utilities \eqref{logut}, associated with driver $F$ defined by \eqref{driver_log}. When $F$ does not depend on $z$, taking the conditional expectation of \eqref{rdtBSDE} with respect to $\F_{t}$, and evaluating this expression at time $0$, together with the fixed initial condition $Y_{0} = y_{0}$ leads to the following characterization of the ergodic cost.

\begin{lem} \label{linlambd}
Let $(Y, Z, \lambda)$ be the unique solution of equation \eqref{rdtBSDE}, as defined in Theorem \ref{thm:BSDERandomtime}, and with a driver $F$ only depending on $V$. Then the ergodic cost $\lambda$ admits the representation
\begin{eqnarray} \label{MClambda1}
\lambda = \frac{\e \SBRA{\int_{0}^{\tau} F(V_{s}) ds}}{\e \SBRA{\tau}}.
\end{eqnarray}
\end{lem}

\paragraph{Exponential forward utility without constraints} When there are no constraints on the portfolio, that is $\Pi = \R^{d}$, the driver $F_{\exp}$ defined by \eqref{driver_exp} and associated with exponential utilities, is linear in $z$ and given by
\begin{eqnarray} \label{driver_exp_NC}
F_{\exp}(v, z) = - \theta(v)^{\top} z - \frac{1}{2} \NRM{\theta(v)}^{2}.
\end{eqnarray}
The representation result for linear BSDEs can be adapted to our framework in order to obtain a representation of the ergodic cost in this case. Note that the following result can be generalized to any linear generator $F$. 

\begin{prop} \label{proplinexp}
Assuming there are no constraints on the portfolio, the unique solution of the ergodic BSDE with random terminal time \eqref{rdtBSDE}, as defined in Theorem \ref{thm:BSDERandomtime}, and with generator $F_{\exp}$ given by \eqref{driver_exp} is given for all $t \geq 0$ by
\begin{eqnarray} \label{Yseexp}
Y_{t} = \e \SBRA{y_{0} \Gamma_{t, \tau} - \int_{t}^{\tau} \Gamma_{t, s} \PAR{\frac{1}{2} \NRM{\theta(V_{s})}^{2} + \lambda}ds | \F_{t}}, \quad \text{a.s},
\end{eqnarray}
where
\begin{eqnarray}
    d \Gamma_{t, s} = - \Gamma_{t, s} \theta(V_{s})^{\top} dW_{s} \quad\text{and}\quad \Gamma_{t, t} = 1. 
\end{eqnarray}
The ergodic cost $\lambda$ is given by
\begin{eqnarray} \label{lambdalinexp}
\lambda = \displaystyle \frac{1}{\e \SBRA{\int_{0}^{\tau} \Gamma_{0, s}ds}} \PAR{ \e \SBRA{y_{0} \Gamma_{0, \tau} - \frac{1}{2} \int_{0}^{\tau} \Gamma_{0, s} \NRM{\theta(V_{s})}^{2}ds} - y_{0}}.
\end{eqnarray}
\end{prop}

\begin{preuve}
The representation theorem for linear BSDEs (\cite{el2008backward}) can be extended to BSDE with random terminal time. Consider $(Y, Z)$ the unique Markovian solution to \eqref{rdtBSDE} in the sense of Theorem \ref{thm:BSDERandomtime} and define the stopped process $(M_{s}^{\tau})_{s \geq t}$,
\begin{eqnarray}
M_{s}^{\tau} := Y_{s \wedge \tau} \Gamma_{t, s \wedge \tau} - \frac{1}{2} \int_{t}^{s \wedge \tau} \Gamma_{t, u} \PAR{\NRM{\theta(V_{u})}^{2} + \lambda} du.
\end{eqnarray}
An application of Ito's formula to the product $Y_{s \wedge \tau} \Gamma_{t, s \wedge \tau}$ shows that $(M_{s}^{\tau})_{s \geq t}$ is a local martingale. Moreover, $\underset{0 \leq s \leq \tau}{\sup} \, Y_{s \wedge \tau}$ and $\underset{0 \leq s \leq \tau}{\sup} \Gamma_{t, s \wedge \tau}$ belong to $L^{2}$ so that the product $\underset{0 \leq s \leq \tau}{\sup} \, Y_{s \wedge \tau} \times \underset{0 \leq s \leq \tau}{\sup} \Gamma_{t, s \wedge \tau}$ is integrable. The martingale $(M_{s}^{\tau})_{s \geq t}$ is thus uniformly integrable, and its value at time $t$ equals the conditional expectation of its terminal value with respect to $\F_{t}$, from which follows \eqref{Yseexp}.

The actual value of $\lambda$ to recover the solution to the ergodic BSDE with random time horizon \eqref{rdtBSDE} with generator $F_{\exp}$ is then uniquely determined from the initial condition $Y_{0} = y_{0}.$ Evaluating \eqref{Yseexp} at time $t=0$ leads to the ergodic cost \eqref{lambdalinexp}, which concludes the proof.
\end{preuve}

\paragraph{Power forward utility without constraints} For ergodic BSDEs with random terminal time and quadratic generator of the form
\begin{eqnarray} \label{driverCH}
F(v, z) = l(v) + a(v)^{\top} z + \frac{\beta}{2} \NRM{z}^{2},
\end{eqnarray}
the Cole-Hopf transform can be used to recover the linear case and, in turn, a representation of the ergodic cost.

\begin{prop} \label{thm:colehopf}
Consider the ergodic BSDE with random terminal time \eqref{rdtBSDE} with generator $F$ given by \eqref{driverCH}
such that $l$ and $a$ are bounded. This equation admits a unique Markovian solution $(y, z, \lambda)$, where $y$ is sublinear, $z$ is bounded and $Y_{0} = y_{0}$, which is given for all $t \geq 0$, by
\begin{eqnarray}
    Y_{t} = \frac{1}{\beta} \ln \e \SBRA{e^{\beta y_{0}} \Gamma_{t, \tau}(\lambda) | \F_{t}}, \quad \text{a.s},
\end{eqnarray}
with
\begin{eqnarray*}
d \Gamma_{t, s}(\lambda) = \Gamma_{t, s}(\lambda) \PAR{ \beta(l(V_{s}) - \lambda)ds + a(V_{s})^{\top} dW_{s}} \quad\text{and}\quad
\Gamma_{t, t}(\lambda) = 1. 
\end{eqnarray*}
Moreover, the ergodic cost $\lambda$ is characterized as
\begin{eqnarray}
\lambda = \underset{\tilde\lambda \in \R}{\argmin} \ABS{\e \SBRA{\Gamma_{0, \tau}}(\tilde\lambda) - 1}.
\end{eqnarray}
\end{prop}

\begin{preuve}
Let $(Y, Z, \lambda)$ be the solution of the ergodic BSDE with generator $F$ given by \eqref{driverCH} such that $Y_t = y(V_t) + y_0 - y(V_0)$, with $y$ is sublinear and $y(0)= 0$, and $Z$ is bounded. Let $P_{s} := e^{\beta Y_{s}}$ and $Q_{s}: = \beta P_{s} Z_{s}$. An application of Ito formula leads to
\begin{eqnarray}
P_{t} = P_{\tau} + \int_{t}^{\tau} \SBRA{\beta P_{s} \PAR{l(V_{s}) - \lambda} + a(V_{s})^{\top}Q_{s}}ds - \int_{t}^{\tau} Q_{s}^{\top}dW_{s}. \label{colhopgen}
\end{eqnarray}
For any fixed $\lambda \in \R$, working as in the proof of Proposition \ref{proplinexp}, the representation theorem for linear BSDEs ensures that the above equation with terminal condition $P_{\tau} = e^{\beta y_{0}}$ admits a unique solution $(P, Q) \in \Sr^{2}(\gamma, \tau) \times \Hr^{2}(\gamma, \tau)$, such that for all $t \geq 0$, 
\begin{eqnarray} \label{semiexplpreuve}
P_{t} = \e \SBRA{e^{\beta y_{0}} \Gamma_{t, \tau}(\lambda) | \F_{t}}, \quad \text{a.s},
\end{eqnarray}
where
\begin{eqnarray*}
d \Gamma_{t, s}(\lambda) &=& \Gamma_{t, s}(\lambda) \PAR{\beta (l(V_{s}) - \lambda)ds + a(V_{s})^{\top}dW_{s}} \\
\Gamma_{t, t}(\lambda) &=& 1.
\end{eqnarray*}
Equation \eqref{semiexplpreuve} at time zero gives $\e \SBRA{\Gamma_{0, \tau}(\lambda)} = 1.$ The map $\lambda \mapsto \e \SBRA{\Gamma_{0, \tau}(\lambda)}$ being strictly monotonic, the ergodic cost $\lambda$ is characterized as
\begin{eqnarray}
\lambda = \underset{\tilde\lambda \in \R}{\argmin} \ABS{\e \SBRA{\Gamma_{0, \tau}(\tilde\lambda)} - 1}.
\end{eqnarray}
\end{preuve}

By construction, the ergodic cost satisfies $\ABS{\lambda} \leq K$ with $K$ given in Assumption \ref{Fgrowthass} (see the proof of Proposition $3.1$ in \cite{liang2017representation}). Hence, the minimum in the above result must be attained in the interval $\SBRA{-K, K}$. Proposition \ref{thm:colehopf} can be then applied to obtain a representation for the ergodic cost $\lambda$ in the case of power forward utilities, with generator $F$ given by \eqref{driver_zar} in the absence of portfolio constraints.

\begin{coro} \label{linquadr}
The unique Markovian solution to the ergodic BSDE with random time horizon \eqref{rdtBSDE} and with generator $F$ given by
\begin{eqnarray}
F(V_{t}, Z_{t}) = \frac{1}{2} \frac{\delta}{1 - \delta} \NRM{Z_{t} + \theta(V_{t})}^{2} + \frac{1}{2} \NRM{Z_{t}}^{2},
\end{eqnarray}
is given, for all $t \geq 0$, by
\begin{eqnarray} \label{semiexplpow}
Y_{t} = y_{0} + (1 - \delta) \ln  \e \SBRA{ \Gamma_{t, \tau}(\lambda) | \F_{t}}, \quad \text{a.s.}
\end{eqnarray}
where $\Gamma_{t, \tau}$ is the value in $\tau$ of the unique solution of the forward equation
\begin{eqnarray} \label{Gammadyn}
d \Gamma_{t, s}(\lambda) &=& \Gamma_{t, s}(\lambda) \PAR{ \frac{1}{1 - \delta}\big(\frac{\delta}{2(1 - \delta)} \NRM{\theta(V_{s})}^{2} - \lambda\big)ds + \frac{\delta}{1 - \delta} \theta(V_{s})^{\top}dW_{s}} \\ 
\Gamma_{t, t}(\lambda) &=& 1 \nonumber.
\end{eqnarray}
Moreover, the ergodic cost $\lambda$ satisfies
\begin{eqnarray} \label{minlambda}
\lambda = \underset{\tilde\lambda \in \SBRA{-K, K}}{\argmin} \ABS{\e \SBRA{\Gamma_{0, \tau}(\tilde\lambda)} - 1}
\end{eqnarray}
\end{coro}

\section{Approximation of the ergodic BSDE} \label{section:numerical}

In this section, we introduce the Euler approximation for the stochastic factor $V$, with the estimation of the horizon time $\tau$ using this continuous approximation. We then present a backward scheme for ergodic BSDE using the representation introduced in Section \ref{section:connection} and investigate the associated discrete-time approximation error.

\subsection{Euler scheme approximation of the stochastic factor}
First, the integral form of the stochastic factor process $\eqref{stochfact}$ gives
\begin{eqnarray} \label{intV}
V_{t} = v_{0} + \int_{0}^{t} \mu(V_{s})ds + \int_{0}^{t} \kappa dW_{s}, \quad t \geq 0.
\end{eqnarray}
Consider a discretization of $\R^{+}$ with constant time step $h$, generating a grid $\pi = \BRA{t_{0} = 0, t_{1}, ... \,  }$. Denoting $\Delta W_{i} = W_{t_{i+1}} - W_{t_{i}}$ the Brownian increments between times $t_{i}$ and $t_{i+1}$, the Euler discretization of $V$ on the time grid $\pi$ is given, for all $i \geq 0$,
\begin{eqnarray}\label{discV}
\overline{V}_{t_{i+1}} &=& \overline{V}_{t_{i}} + \mu(\overline{V}_{t_{i}}) h + \kappa \Delta W_{i}, \\
\overline{V}_{0} &=& v_0. \nonumber
\end{eqnarray}
In the sequel, we will consider the continuous Euler scheme associated to \eqref{discV} on grid $\pi$, defined by
\begin{eqnarray} \label{eulerV}
\overline{V_{t}} = v_{0} + \int_{0}^{t} \mu(\overline{V}_{s^{-}})ds + \int_{0}^{t} \kappa dW_{s}, \quad t \geq 0, 
\end{eqnarray}
where $s^{-} = \max \BRA{t_{i} \in \pi, \, t_{i} \leq s}$.\\ 
We define the approximation of the stopping time $\tau$ as follows
\begin{eqnarray} \label{bartau}
\overline{\tau} = \inf \BRA{t \geq T, \, \overline{V}_{t} = v_{0}}.
\end{eqnarray}
This stopping time is well defined since the Euler discretization $(\overline{V_{t}})_{t \geq 0}$ is also ergodic, (see \cite{talay1990second}, \cite{mattingly2002ergodicity}). The horizon $\bar{\tau}$ defined by \eqref{bartau} we consider can be written as an exit time of the diffusion of a smooth domain. In fact,
\begin{eqnarray}
\bar{\tau} = \inf \BRA{t \geq T, \, \overline{V}_{t} \notin \left]-\infty, v_{0} \right[} \ind_{V_{T} < v_{0}} + \inf \BRA{t \geq T, \, \overline{V}_{t} \notin \left]v_{0}, +\infty \right[} \ind_{V_{T} > v_{0}}, \quad \p\, \text{a.s}.
\end{eqnarray}
Results from \cite{geiss2017} and \cite{matoussi2016numerical} related to the estimation of $\tau$ with an Euler scheme can be applied in our setting. We recall the following result from Theorem 3.9 \cite{geiss2017}.

\begin{prop} \label{L1errtau}
Let Assumption \ref{weakdissass} hold and assume that there exists $ q\in[4,\infty[$ such that under the notation of Lemma \ref{lem:expint}, $\frac{q}{q-1} 6 C_{\mu} < \frac{1}{4 B^{+}} \wedge \frac{1}{4 B^{-}}$. Then, there exists a constant $C > 0$ such that the error in $L^{1}$ of the approximation of the return time $\tau$ with the continuous Euler scheme satisfies
\begin{eqnarray}\label{esttau}
\e \SBRA{\ABS{\tau - \bar{\tau}}} \leq C h^{1/2}.
\end{eqnarray}
\end{prop}

\begin{preuve}
Let $\beta \in \R$ be such that $\frac{q}{q-1} 6 C_{\mu} < \beta < \frac{1}{4 B^{+}} \wedge \frac{1}{4 B^{-}}$. Under this assumption, Lemma \ref{lem:expint} yields that the random time $\tau$ admits exponential moment of order $\beta$. An application of Markov's inequality then gives, that, for any $k \in \n$;
\begin{eqnarray}
 \p (\tau \geq k) \leq \e \SBRA{e^{\beta \tau}} e^{- \beta k}.   
\end{eqnarray}
Applying Theorem $3.9$ from \cite{geiss2017} we conclude.
\end{preuve}

\subsection{Discrete-time approximation error} \label{discrerrsect}

Next, we recall the usual time discretization of BSDEs, applied to the ergodic BSDE with random time horizon \eqref{rdtBSDE}. Consider the time discretization $\pi = \BRA{ 0 = t_{0}, t_{1}, ...}$ of $\R^{+}$, with time step $h$ and the forward Euler scheme $\overline{V}$ for the stochastic factor $V$ defined by \eqref{eulerV}. Let $\bar{\lambda}$ denote an approximation of the ergodic cost $\lambda$. It can be estimated either by a Monte Carlo approximation or as a trainable parameter of the deep learning algorithm developed in the next section. Then starting from $\overline{Y}_{\bar{\tau}} = y_{0}$, we define the discrete time process $(\overline{Y}, \overline{Z})$ on $\pi$, for $i < \frac{\bar{\tau}}{h}$,
\begin{eqnarray}
\overline{Z}_{t_{i}} &=& \frac{1}{h} \e \SBRA{\overline{Y}_{t_{i+1}} \Delta W_{t_{i}} | \F_{t_{i}}} \label{discrZ} \\
\overline{Y}_{t_{i}} &=& \e \SBRA{\overline{Y}_{t_{i+1}} | \F_{t_{i}}} + \ind_{t_i\leq \bar{\tau}}h \SBRA{F(\overline{V}_{t_{i}}, \overline{Z}_{t_{i}}) - \overline{\lambda}}, \label{discrY}
\end{eqnarray}
One may check that under Assumption \ref{weakdissass}, for all $i$, $(\overline{Y}_{t_{i}}, \overline{Z}_{t_{i}}) \in L^{2}$. Moreover, as the process $Z$ is bounded, we may equivalently work with a Lipschitz driver $F \circ \varphi_{Z_{\max}}$, where $\varphi_{Z_{\max}}$ is the projection on the centered ball of $\R^{d}$ of radius $Z_{max}.$ As mentioned in \cite{bender2012least} and \cite{chassagneux2016numerical}, this may lead to numerical difficulties if this bound $Z_{\max}$ is too large. 

\vspace{0.3cm}
For what follows, it will be convenient to work with a continuous extension of $\overline{Y}$ in $\Sr^{2}$. This is possible since from the martingale representation theorem, there exists a process $\tilde{Z} \in \Hr^{2}$ such that
\begin{eqnarray} \label{zmrt}
    \overline{Y}_{t_{i+1}} - \e \SBRA{\overline{Y}_{t_{i+1}} | \F_{t_{i}}} = \int_{t_{i}}^{t_{i+1}} \tilde{Z}_{s}^{\top}dW_s,
\end{eqnarray}
which allows to consider the continuous extension of $(Y_{t_{i}})_{i < \frac{\bar{\tau}}{h}}$ on $[0, \bar{\tau}]$
\begin{eqnarray}
\overline{Y}_{t} = \overline{Y}_{\bar{\tau}} + \int_{t}^{\bar{\tau}} F(\overline{V}_{s^{-}}, \overline{Z}_{s^{-}})ds - (\bar{\tau} -t) \overline{\lambda} - \int_{t}^{\bar{\tau}} \tilde{Z}_{s}^{\top} dW_{s}.
\end{eqnarray}
Finally, we will also consider the approximation in $L^{2}$ of $Z$ by a process constant on each time interval $[t_{i}, t_{i+1}]$, defined by
\begin{eqnarray}
\hat{Z}_{t_{i}} = \frac{1}{h} \e \SBRA{\int_{t_{i}}^{t_{i+1}} Z_{s}ds | \F_{t_{i}}}.
\end{eqnarray}

\begin{rmq}
From Ito's isometry and \eqref{zmrt} we have, for all $i \geq 0$
    \begin{eqnarray}
        \overline Z_{t_{i}} = \frac{1}{h} \e \SBRA{\int_{t_{i}}^{t_{i+1}} \tilde{Z}_{s}ds | \F_{t_{i}}}.
    \end{eqnarray}
\end{rmq}

Now, we next provide a bound for the square of the discrete time approximation error, up to a stopping time $\theta$
\begin{eqnarray} \label{deferrconv}
\Err(h)_{\theta}^{2} = \underset{i}{\max}\, \e \SBRA{\underset{t \in \SBRA{t_{i}, t_{i+1}}}{\sup} \ind_{t \leq \theta} \ABS{Y_{t} - \overline{Y}_{t_{i}}}^{2}} + \e \SBRA{\int_{0}^{\theta} \NRM{Z_{t} - \overline{Z}_{t^{-}}}^{2}dt}, 
\end{eqnarray}
where $t^{-} = \sup \BRA{s \in \pi, s \leq t}.$ We will control this error through this error quantity
\begin{eqnarray}
     \Rr(Z)_{\Hr^{2}}^{\pi} = \e \SBRA{\int_{0}^{\tau} \NRM{Z_{t} - \hat{Z}_{t^{-}}}^{2}dt}.
\end{eqnarray}
From now on, $C$ denotes a generic constant, which depends on $V_0$, $C_v$ and $C_z$.
\begin{rmq} \label{inegconvZ}
\begin{enumerate}
\item Let $\theta$ be an $\F_{t}$ stopping time in grid $\pi$. We may control the error term\newline $\e \SBRA{\int_{0}^{\theta} \NRM{Z_{s} - \tilde{Z_{s}}}^{2}ds}$ which will provide the desired bound on $\Err(h)$. In fact, using Jensen's inequality, one can show (see \cite{bouchard2009strong}) that for any stopping time $\theta$ in the time grid $\pi$,
\begin{eqnarray}
\e \SBRA{\int_{0}^{\theta} \NRM{Z_{s} - \overline{Z}_{s^{-}}}^{2}ds} \leq C \PAR{\e \SBRA{\int_{0}^{\theta} \NRM{Z_{s} - \tilde{Z}_{s}}^{2}ds} + \e \SBRA{\int_{0}^{\theta} \NRM{Z_{s} - \hat{Z}_{s^{-}}}^{2}ds}}.
\end{eqnarray}
\item We adapt the proof of Theorem 3.2 in \cite{bouchard2009discrete} to the case of unbounded time horizon to show that 
\begin{eqnarray}\label{boundZ}
    \Rr(Z)_{\Hr^{2}}^{\pi}\leq C h.
\end{eqnarray}

In the first instance, we can show inequality \eqref{boundZ} for any bounded stopping time $\tau\wedge T$. Then, using the fact that Z is bounded and the convergence theorem, we can obtain the result for any stopping time $\tau$.
\end{enumerate}
\end{rmq}
Controlling the error \eqref{deferrconv} implies that we control the error in $\Sr^{2} \times \Hr^{2}$ of the discrete time approximation $(\overline{Y}_{t_{i}}, \overline{Z}_{t_{i}})$ for all $i$. We provide a bound for the discretization error of a backward scheme for ergodic BSDE in terms of the quantities $\Rr(Z)_{\Hr^{2}}^{\pi}$, $\ABS{\lambda - \bar{\lambda}}$ and $\ABS{\tau - \bar{\tau}}$. We refer to \cite{bouchard2009strong}, whose results concern the approximation of classical BSDE with random terminal time.

\begin{prop}
Let \eqref{lemma:Kz} and Assumptions \ref{weakdissass}, \ref{Fgrowthass} hold. Also assume that the condition of Proposition \ref{L1errtau} is satisfied. Then, there exists a constant $C > 0$ such that
\begin{eqnarray} \label{errT}
\Err(h)_{\tau \vee \bar{\tau}}^{2} \leq\Err(h)_{\tau^{+} \vee \bar{\tau}^{+}}^{2} \leq C \left(h^{1/2}  + \ABS{\lambda - \bar{\lambda}}^{2}  \right),
\end{eqnarray}
where $\tau^+$ is the next time in the grid $\pi$ after $\tau$, $\tau^+:=\inf\{t\in\pi : \tau\leq t\}.$
\end{prop}

In addition to the usual spatial error for discretization schemes for finite horizon BSDE, we get here a term related to the estimation of the return time $\tau$ with the Euler approximation $\overline{V}$ as well as an error term related to the estimation of the ergodic constant $\lambda$. 

\begin{proof}
We follow the arguments described in \cite{bouchard2009strong}, and add the error term specific to our ergodic equation $\ABS{\lambda - \bar{\lambda}}$. Let $\theta$ be an $\F_{t}$ stopping time in the grid $\pi$. Applying Ito's lemma to $(Y - \overline{Y})^{2}$, between $t \wedge \theta$ and $t_{i+1} \wedge \theta$, for any time $t \in \SBRA{t_{i}, t_{i+1}}$
\begin{eqnarray}
\Delta_{t,t_{i+1}}^{\theta}&:= &\e \SBRA{ \ABS{Y_{t \wedge \theta} - \overline{Y}_{t \wedge \theta}}^{2} + \int_{t \wedge \theta}^{t_{i+1} \wedge \theta} \NRM{Z_{s} - \tilde{Z_{s}}}^{2}ds} \nonumber\\
    &= &\e \SBRA{\ABS{Y_{t_{i+1} \wedge \theta} - \overline{Y}_{t_{i+1}\!\wedge \theta}}^{2}} + \e \SBRA{2 \int_{t \wedge \theta}^{t_{i+1} \wedge \theta}\!\! \!(Y_{s} - \overline{Y}_{s})(\ind_{s < \tau}(F(V_{s}, Z_{s}) - \lambda) - \ind_{s < \bar{\tau}} (F(\overline{V}_{s^{-}}, \overline{Z}_{s^{-}}) + \bar{\lambda}))ds} \nonumber \\
    &= &\e \SBRA{\ABS{Y_{t_{i+1} \wedge \theta} - \overline{Y}_{t_{i+1} \wedge \theta}}^{2}} + \e \left[ 2 \int_{t \wedge \theta}^{t_{i+1} \wedge \theta} (Y_{s} - \overline{Y}_{s}) \ind_{s \leq \bar{\tau}} (F(V_{s}, Z_{s}) - \lambda - F(\bar{V}_{s^{-}}, \bar{Z}_{s^{-}}) + \bar{\lambda})ds \right] \nonumber\\
    &&  \qquad  + \e \left[ 2 \int_{t \wedge \theta}^{t_{i+1} \wedge \theta} (Y_{s} - \overline{Y}_{s}) (\ind_{s \leq \tau} - \ind_{s \leq \bar{\tau}})(F(V_{s}, Z_{s}) - \lambda) ds \right].
\end{eqnarray}
Using the inequality $2ab \leq \alpha a^{2} + \frac{1}{\alpha}b^{2}$, for a suitable $\alpha > 0$ to be chosen later, we obtain
\begin{eqnarray*}
\Delta_{t,t_{i+1}}^{\theta} &\leq& \e \SBRA{\ABS{Y_{t_{i+1} \wedge \theta} - \overline{Y}_{t_{i+1} \wedge \theta}}^{2}} + \alpha \e \SBRA{\int_{t \wedge \theta}^{t_{i+1} \wedge \theta} \ABS{Y_{s} - \overline{Y}_{s}}^{2}ds} \\
    && + \frac{2}{\alpha} \e \SBRA{\int_{t \wedge \theta}^{t_{i+1} \wedge \theta} \ind_{s < \bar{\tau}} (F(V_{s}, Z_{s}) - F(\overline{V}_{s^{-}}, \overline{Z}_{s^{-}}))^{2}ds} + \frac{2}{\alpha}\e \SBRA{\int_{t \wedge \theta}^{t_{i+1} \wedge \theta} \ind_{s < \bar{\tau}}\ABS{\lambda - \bar{\lambda}}^{2}ds} \\
    && + \frac{2}{\alpha} \e \SBRA{\int_{t \wedge \theta}^{t_{i+1} \wedge \theta} \ind_{\tau \leq s < \bar{\tau}} (F(V_{s}, Z_{s}) - \lambda)^{2}ds + \int_{t \wedge \theta}^{t_{i+1} \wedge \theta} \ind_{\bar{\tau} \leq s < \tau} (F(V_{s}, Z_{s}) - \lambda)^{2}ds}.
\end{eqnarray*}
On the event $\BRA{s > \tau}$, we have $Y_{s} = Y_{\tau}$ so that $Z_{s} = 0$. Then, using the Lipschitz properties of the driver $F$ \eqref{Flipschv}, \eqref{Flipschz}, the boundedness of $Z$, Remark \ref{inegconvZ} and results on the Euler approximation of $V$ we obtain,
\begin{eqnarray*}
\Delta_{t,t_{i+1}}^{\theta} &\leq& \e \SBRA{\ABS{Y_{t_{i+1} \wedge \theta} - \overline{Y}_{t_{i+1} \wedge \theta}}^{2}} + \alpha \e \SBRA{\int_{t \wedge \theta}^{t_{i+1} \wedge \theta} \ABS{Y_{s} - \overline{Y}_{s}}^{2}ds} \\
    && + \frac{C}{\alpha} \e \SBRA{\int_{t \wedge \theta}^{t_{i+1} \wedge \theta\wedge \bar{\tau}} \PAR{h + \NRM{Z_{s} - \hat{Z}_{s^{-}}}^{2} + \NRM{{Z}_{s} - \tilde{Z}_{s}}^{2}}ds} \\
    && + \frac{C}{\alpha} \e \SBRA{  \int_{t \wedge \theta}^{t_{i+1} \wedge \theta}  \ind_{\tau  \leq s \leq \bar{\tau}}K^{2} + \ind_{\bar{\tau} \leq s \leq \tau} Z_{\max}^{2} + \ind_{\tau \wedge \bar{\tau} \leq s \leq \tau \vee \bar{\tau}} \lambda^{2} ds} \\
    && +\frac{2}{\alpha}\e \SBRA{\int_{t \wedge \theta}^{t_{i+1} \wedge \theta\wedge \bar{\tau}} \ABS{\lambda - \bar{\lambda}}^{2}ds} .
\end{eqnarray*}
 In turn, Gronwall's lemma gives
\begin{eqnarray} \label{transinegy}
\e \SBRA{\ABS{Y_{t \wedge \theta} - \overline Y_{t \wedge \theta}}^{2}} &\leq& \Delta_{t,t_{i+1}}^{\theta}
\nonumber\\
&\leq& (1 + C_{\alpha} h) \e \SBRA{\ABS{Y_{t_{i+1} \wedge \theta} - \overline{Y}_{t_{i+1} \wedge \theta}}^{2}} \nonumber \\
    && + \PAR{C_{\alpha} h + \frac{C}{\alpha}} \e \SBRA{\int_{t \wedge \theta}^{t_{i+1} \wedge \theta} \PAR{h + \NRM{Z_{s} - \hat{Z}_{s^{-}}}^{2} + \NRM{Z_{s} - \tilde{Z}_{s}}^{2}}ds} \\
    && + \PAR{C_{\alpha} h + \frac{C}{\alpha}} \e \SBRA{  \int_{t \wedge \theta}^{t_{i+1} \wedge \theta} \ind_{\tau \wedge \bar{\tau} \leq s \leq \tau \vee \bar{\tau}} \max(K^{2}, Z_{\max}^{2}, \lambda^{2}) ds} \nonumber \\
    && + \PAR{C_{\alpha} h + \frac{2}{ \alpha}} \e \SBRA{\int_{t \wedge \theta}^{t_{i+1} \wedge \theta\wedge \bar{\tau}} \ABS{\lambda - \bar{\lambda}}^{2}ds}.  
    \nonumber
\end{eqnarray}
Recall that the ergodic constant $\lambda$ is bounded, and, thus, the maximum in the third line above is finite. Substituting $t=t_{i}$ in \eqref{transinegy}, taking $\alpha > 0$ sufficiently large and the time step $h$ small enough, we obtain
\begin{eqnarray*}
 \e \SBRA{ \ABS{Y_{t_{i} \wedge \theta} - \overline{Y}_{t_{i} \wedge \theta}}^{2}} &+& \e \SBRA{\int_{t_i \wedge \theta}^{t_{i+1} \wedge \theta}\!\!  \NRM{Z_{s} - \tilde{Z_{s}}}^{2}ds} \leq (1 + C h) \e \SBRA{\ABS{Y_{t_{i+1} \wedge \theta} - \overline{Y}_{t_{i+1} \wedge \theta}}^{2}} \nonumber\\
 &+& C \e \SBRA{\int_{t_i \wedge \theta}^{t_{i+1} \wedge \theta} \!\!\PAR{h + \NRM{Z_{s} - \hat{Z}_{s^{-}}}^{2}+ \NRM{Z_{s} - \tilde{Z}_{s}}^{2}}ds} \\
    & +& C \e \SBRA{  \int_{t_i \wedge \theta}^{t_{i+1} \wedge \theta} \!\!\ind_{\tau \wedge \bar{\tau} \leq s \leq \tau \vee \bar{\tau}} ds}  + C \e \SBRA{\int_{t_i \wedge \theta}^{t_{i+1} \wedge \theta\wedge \bar{\tau}} \ABS{\lambda - \bar{\lambda}}^{2}ds}   .
\end{eqnarray*}
Summing on $i$  we have
\begin{eqnarray} \label{finalineqErr}
\underset{i}{\max}\, \e \SBRA{ \ABS{Y_{t_{i} \wedge \theta} - \overline{Y}_{t_{i} \wedge \theta}}^{2}} + \e \SBRA{\int_{0}^{\theta} \NRM{Z_{s} - \tilde{Z_{s}}}^{2}ds} \!\!&&\leq C \left( \e \SBRA{\ABS{Y_{\theta} - \overline{Y}_{\theta}}^{2}} + h\e\SBRA{\theta}  + \Rr(Z)_{\Hr^{2}}^{\pi} \right. \nonumber \\
&& \left. + \e \SBRA{\ABS{\theta\wedge(\bar{\tau} \vee \tau) - \tau \wedge \bar\tau}} + \ABS{\lambda - \bar{\lambda}}^{2} \e \SBRA{\theta\wedge\bar\tau} \right).\nonumber\\
&&
\end{eqnarray}
Then, it follows, again by using Remark \ref{inegconvZ} that
\begin{eqnarray}
\Err(h)_{\theta}^{2}&&\leq C \left( \e \SBRA{\ABS{Y_{\theta} - \overline{Y}_{\theta}}^{2}} + h\e\SBRA{\theta}  + \Rr(Z)_{\Hr^{2}}^{\pi}+ \e \SBRA{\ABS{\theta\wedge(\bar{\tau} \vee \tau) - \tau \wedge \bar\tau}}  + \ABS{\lambda - \bar{\lambda}}^{2} \e \SBRA{\theta\wedge\bar\tau} \right).\nonumber\\
&&
\end{eqnarray}
For the stopping time $\theta = \tau^{+} \vee \bar{\tau}^{+}$ in the time grid $\pi$, one observes that $Y_{\tau^{+} \vee \bar{\tau}^{+}} - \overline{Y}_{\tau^{+} \vee \bar{\tau}^{+}} = 0$, and we obtain that
\begin{eqnarray*}
\Err(h)_{\tau^{+} \vee \bar{\tau}^{+}}^{2} \leq  C \left( h\e\SBRA{\tau^{+} \vee \bar{\tau}^{+}}  + \Rr(Z)_{\Hr^{2}}^{\pi}+ \e \SBRA{\ABS{ \tau -  \bar\tau}}  + \ABS{\lambda - \bar{\lambda}}^{2} \e \SBRA{\bar\tau} \right).
\end{eqnarray*}
Finally, combining the estimates \eqref{boundZ}, \eqref{esttau} and using the inequalities $\tau^+\leq h+\tau$, $\bar{\tau}^+\leq \bar \tau+h$, we conclude.
\end{proof}

\section{Deep learning algorithms for the simulation of ergodic BSDEs and forward utilities}  \label{section:simulation}

In this section, we introduce new deep learning algorithms for the simulation of ergodic BSDEs, based on the representation \eqref{rdtBSDE} of Markovian solutions using BSDEs with random time horizon. The first neural network based algorithm for solving BSDEs was initially proposed in \cite{han2017deep}. Since then, there has been a growing interest in developing deep learning algorithms solving BSDEs with finite horizons. We introduce here two algorithms solving the ergodic BSDE \eqref{ebsdeInitCond}, which can be seen as the ergodic counterpart of the neural networks algorithms introduced in \cite{han2017deep} and \cite{kapllani2020deep}.
The algorithms introduced below approximate the ``Markovian'' solution $(y(V_t) + y_0 - y(V_0), z(V_t), \lambda)_{t\geq 0 }$ to the ergodic BSDE with random time horizon \eqref{rdtBSDE},
\begin{eqnarray}\label{FBSDE}
dV_{t} &=& \mu(V_{t}) + \kappa dW_{t}, \quad V_{0} = v_{0}, \nonumber \\
Y_{t} &=& Y_{\tau} + \int_{t}^{\tau} F(V_{s}, Z_{s})ds - \lambda (\tau- t) - \int_{t}^{\tau} Z_{s}^{\top}dW_{s}, \\
Y_{\tau} &=& Y_0 =  y_{0}, \nonumber
\end{eqnarray}
We can remark from the system \eqref{FBSDE} that
\begin{enumerate}
\item The initial value $Y_{0}$ is a known quantity, equal to $y(v_{0}) = y_{0}$. This allows us to use a forward scheme starting from $y_{0}$. 
\item The recurrence property of the stochastic factor $V$ also provides a known terminal condition to \eqref{rdtBSDE}. Indeed, by definition of the return time $\tau$, $Y_{\tau} = Y_0 $. This allows us to define the loss functions.
\item The functions $y$ and $z$ are only functions of the stochastic factor $V$ and do not depend on time, as it is the case for standard BSDEs with fixed time horizon.
\end{enumerate}

\subsection{Deep learning algorithm for the simulation of ergodic BSDE}

In this section, we present our two main algorithms for the simulation of ergodic BSDEs, called GeBSDE and LAeBSDE. The first neural network based algorithm for solving BSDEs was initially proposed in \cite{han2017deep}. In the context of ergodic BSDEs with random terminal time of type \eqref{rdtBSDE}, the initial value $Y_{0} = y_{0}$ is known and will, thus, not be learned. We instead approximate the ergodic cost $\lambda$ as a trainable parameter of the model. We present two algorithms:
\begin{itemize}[-]
    \item A global solver GeBSDE which consists in the minimization of a square loss function at the random horizon $\tau$ and which is the adaptation of the Deep BSDE solver of \cite{han2017deep} in the case of ergodic BSDEs.
    \item A locally additive solver LAeBSDE optimized according to the aggregation of local loss functions up to the random horizon, inspired from the deep backward multi-step introduced in \cite{germain2021neural} and the LaBSDE solver in \cite{kapllani2020deep}.
\end{itemize}

Using the same notation as in the previous section, the forward stochastic factor $V$ is approximated by a Euler discretization on the time grid $\pi$. Denoting for all $i \geq 0$, $\Delta W_{t_{i}} = W_{t_{i+1}} - W_{t_{i}}$ the Brownian increment at time $t_{i}$, we write
\begin{eqnarray*}
\overline{V}_{t_{i+1}} &=& \overline{V_{t_{i}}} + \mu(\overline{V}_{t_{i}}) h + \kappa \Delta W_{i}, \\
\overline{V_{0}} &=& v_{0}. \nonumber
\end{eqnarray*}
We denote by $\tilde{\tau}$ the first hitting time in the time grid of $\overline{V}$ to $v_{0}$ after $T$,
\begin{eqnarray}
    \tilde{\tau} = \inf \BRA{ t_{i} > T, \, t_{i} \in \pi \, ; (\overline{V}_{T} - v_{0})(\overline{V}_{t_{i}} - v_{0}) \leq 0},
\end{eqnarray}
where we choose $T \in \pi$.

\paragraph{GeBSDE solver}
Starting from initial value $Y_{0} = y_{0}$, we consider a forward discretization of the equation on the time grid $\pi$ with constant time step $h$. 
 The process $Z_{t}=z(V_t)$ at time $t_{i}$ is represented by a neural network $\Z^{\theta}: \R\to \R^d$, function of $\overline{V_{t_{i}}}$ and with parameter $\theta.$ 
The approximation $\overline{Y}_{t_{i}}^{\theta, \bar{\lambda}}$ of $Y_{t_i}$ depends on the optimization parameter $\theta$ as well as on the trainable parameter $\bar{\lambda}$ through the forward discretization of the ergodic BSDE
\begin{eqnarray} \label{discrebsde}
\overline{Y}_{t_{i+1}}^{\theta, \bar{\lambda}} = \overline{Y}_{t_{i}}^{\theta, \bar{\lambda}} - F(\overline{V}_{t_{i}}, \Z^{\theta}(\overline{V_{t_{i}}})) h + \bar{\lambda} h + \Z^{\theta}(\overline{V_{t_{i}}}) \Delta W_{t_{i}}.
\end{eqnarray}
The output $\overline{Y}_{\bar{\tau}}^{\theta, \bar{\lambda}}$ aims to match the terminal value $Y_{\tau} = y(V_{0}) = y_{0}$, by minimizing over parameters $(\theta, \bar{\lambda})$ the expected square loss function
\begin{eqnarray}\label{mslf}
L_{g}(\theta, \bar{\lambda}) = \e \SBRA{\ABS{y_{0} - \overline{Y}_{\bar{\tau}}^{\theta, \bar{\lambda}}}^{2}}.
\end{eqnarray}
The loss function \eqref{mslf} is approximated by the empirical loss function over a batchsize $B$,
\begin{eqnarray} \label{emplossfunc}
L_{g}^{B}(\theta, \bar{\lambda}) = \displaystyle\frac{1}{B} \sum_{j=1}^{B} \ABS{y_{0} - \overline{Y}_{\tilde{\tau_{j}}}^{\theta, \bar{\lambda}, j}}^{2}.
\end{eqnarray}
Finally, we denote $M$ the number of gradient descent performed in the optimization.

\begin{algorithm}[H]
\DontPrintSemicolon
Let $\Z^{\theta}$ be a neural network defined on $ \R$, valued in $\R^{d}$, with parameters $\theta.$ Let $\bar{\lambda}^{0} \in \R$ be the initialisation of the trainable parameter representing the ergodic cost. Define $N_{T} = \lfloor \frac{T}{h} \rfloor + 1$. \\
\For{$j=1, ..., B$}{
\For{$k\in \BRA{0, ..., N_{T} + 1}$, starting from $\overline{V}_{0}^{j} = v_{0}$}{
Sample $\Delta W_{t_{k}}^{j}$ from a Gaussian vector. \\
$\overline{V}_{t_{k+1}}^{j} = \overline{V}_{t_{k}}^{j} + \mu( \overline{V}_{t_{k}}^{j}) h + \kappa^{\top} \Delta W_{t_{k}}^{j},$ \\
}
Let $N_{j} = N_{T}+1$. \\
\While{$(\overline{V}_{t_{N_{T}}}^{j} - v_{0})(\overline{V}_{t_{N_{j}}}^{j} - v_{0}) > 0$}{
Sample $\Delta W_{t_{k}}^{j}$ from a Gaussian vector. \\
$\overline{V}_{t_{N_{j}+1}}^{j} = \overline{V}_{t_{N_{j}}}^{j} + \mu( \overline{V}_{t_{k}}^{j}) h + \kappa^{\top} \Delta W_{t_{k}}^{j},$ \\
$N_{j} = N_{j}+1$
}
Set, $h N_{j} = \tilde{\tau_{j}}$. \\
}
\For{$m=0$, ..., $M$}{
\For{$j=1, ..., B$}{
\For{$k\in \BRA{0, ..., N_{j}-1}$, starting from $\overline{Y}_{0}^{j} = y_{0}$ }{
$\overline{Y}_{t_{k+1}}^{\theta^{m}, \bar{\lambda}^{m}, j} = \overline{Y}_{t_{k}}^{\theta^{m}, \bar{\lambda}^{m}, j} - h F(\overline{V}_{t_{k}}^{j}, \Z^{\theta^{m}}(\overline{V}_{t_{k}}^{j})) + \bar{\lambda}^{m} h + \Z^{\theta^{m}}(\overline{V}_{t_{k}}^{j})^{\top} \Delta W_{t_{k}},$ \\
}}
Compute $ L^{B}(\theta^{m}, \bar{\lambda}^{m}) = \frac{1}{B} \sum_{j=1}^{B} \ABS{y_{0} - \overline{Y}_{\tilde{\tau}_{j}}^{\theta^{m}, \bar{\lambda}^{m}, j}}^{2}$. \\
Update $\theta^{m+1} = \theta^{m} - \rho_{m} \nabla_{\theta} L^{B}(\theta^{m}, \bar{\lambda}^{m})$ and $\bar{\lambda}^{m+1} = \bar{\lambda}^{m} - \rho_{m} \nabla_{\bar{\lambda}} L^{B}(\theta^{m}, \bar{\lambda}^{m})$.
}
\caption{Global eBSDE Algorithm - (GeBSDE)}
\label{debsdealg}
\end{algorithm}

\paragraph{}
Using the results of \cite{chan2019machine} and \cite{han2017deep}, we develop a neural network consisting in $2$ hidden layers of $20+d$ neurons each, where $d$ is the dimension of the Brownian Motion. For the simulation, we use the $tanh$ activation function and the Adam optimizer with a learning rate $\rho_{m}$ to update both parameters $\theta^{m}$ and $\bar{\lambda}^{m}$. The learning rate parameter can be optimized depending on the example, as investigated in \cite{chan2019machine}. However, choosing a large enough initial learning rate as well as the Adam optimizer, reduce the risk the algorithm gets stuck in a local minimum. Finally, we use a Glorot normal initialization for the parameters of the neural network and the trainable parameter $\bar{\lambda}$, the latter being constrained to be in the interval $\SBRA{-K, K}$.

\paragraph{LAeBSDE solver -} Some other deep learning algorithms rely on a global optimization involving local loss function at each time step, as studied in \cite{hure2020deep} and \cite{kapllani2020deep}. Such algorithms approximate $Y$ with a neural network, while $Z$ can either be computed with automatic differentiation or with another neural network. Numerical results in \cite{hure2020deep} yield that automatic differentiation may lead to some additional errors, so that we will rather use a neural network to approximate $Z$. In the context of ergodic BSDE, the ergodic cost is again approximated as a trainable parameter of the model $\bar{\lambda}$ on which the loss function will depend.

The construction of local loss functions relies on the time discretization \eqref{discrebsde}. In fact, iterating this equation with the initial condition $Y_{0} = y_{0}$, gives for all $i \geq 1$,
\begin{eqnarray} \label{iterative_timediscr}
\overline{Y}_{t_{i}}^{\theta, \bar{\lambda}} = y_{0} - \sum_{k=0}^{i-1} F(\overline{V}_{t_{k}}, \Z^{\theta}(\overline{V_{t_{k}}})) h + \bar{\lambda} h + \Z^{\theta}(\overline{V_{t_{k}}}) \Delta W_{t_{k}}.
\end{eqnarray}
In turn, introducing two neural networks $\Y^{\theta_{1}}$ defined on $\R$ and valued in $\R$ and $\Z^{\theta_{2}}$ defined on $\R$ valued in $\R^{d}$, the local loss function at time $t_{i}$ can be defined as the expected square distance between $\Y^{\theta_{1}}(\overline{V_{t_{i}}})$ and \eqref{iterative_timediscr},
\begin{eqnarray} \label{loc_loss}
L_{\loc, t_{i}}(\theta_{1}, \theta_{2}, \bar{\lambda}) = \e \SBRA{\ABS{\Y^{\theta_{1}}(\overline{V_{t_{i}}}) + \sum_{k=0}^{i-1} F(\overline{V}_{t_{k}}, \Z^{\theta_{2}}(\overline{V_{t_{k}}})) h - \bar{\lambda} h - \Z^{\theta_{2}}(\overline{V_{t_{k}}}) \Delta W_{t_{k}} - y_{0}}^{2}}.
\end{eqnarray}
The loss function is then constructed by summing those local loss functions over $i$. In our framework, $\tau$ is not necessarily bounded and thus, this sum could have infinitely many terms. However, when approximating expectations over a batchsize $B$, one can express this empirical loss function as a sum up to time $\underset{j \in B}{\max} \, \tilde{\tau_{j}}.$ Note that since the terminal time is random, the local loss function at time $t_{i}$ \eqref{loc_loss} is computed on the set of trajectories $\T_{i} = \BRA{j \in \BRA{1, ..., B} \, ; \, \tilde{\tau_{j}} \geq t_{i}}$, for which the approximated return time is larger than $t_{i}.$ The empirical version of \eqref{loc_loss} is given by
\begin{eqnarray}
    L_{\loc, t_{i}}^{B}(\theta_{1}, \theta_{2}, \bar{\lambda}) = \frac{1}{\ABS{\T_{i}}} \sum_{j \in \T_{i}} \ABS{\Y^{\theta_{1}}(\overline{V}_{t_{i}}^{j}) + \sum_{k=0}^{i-1} F(\overline{V}_{t_{k}}^{j}, \Z^{\theta_{2}}(\overline{V}_{t_{k}}^{j})) h - \bar{\lambda} h - \Z^{\theta_{2}}(\overline{V}_{t_{k}}^{j}) \Delta W_{t_{k}}^{j} - y_{0}}^{2}.
\end{eqnarray}
Denoting $N_{\max}^{B} := \frac{1}{h}\underset{j \in B}{\max} \, \tilde{\tau_{j}}$, the index in the grid of the greatest return time over the fixed amount of samples $B$, the empirical locally additive general loss is given by
\begin{eqnarray}
    L^{B}_{\loc}(\theta_{1}, \theta_{2}, \bar{\lambda}) = \sum_{i=1}^{N_{\max}^{B}} L_{\loc, t_{i}}^{B}(\theta_{1}, \theta_{2}, \bar{\lambda}).
\end{eqnarray}

\begin{algorithm}[H]
\DontPrintSemicolon
Let $\Y^{\theta_{1}}$ be a neural network defined on $\R$ and taking values in $\R$ with parameters $\theta_{1}$, and $\Z^{\theta_{2}}$ be a neural network defined on $ \R$, taking values in $\R^{d}$, with parameters $\theta_{2}.$ Let $\bar{\lambda}^{0} \in \R$ be the initialisation of the trainable parameter representing the ergodic cost. Define $N_{T} = \lfloor \frac{T}{h} \rfloor + 1$. \\
\For{$j=1, ..., B$}{
\For{$k\in \BRA{0, ..., N_{T} + 1}$, starting from $\overline{V}_{0}^{j} = v_{0}$}{
Sample $\Delta W_{t_{k}}^{j}$ from a Gaussian vector. \\
$\overline{V}_{t_{k+1}}^{j} = \overline{V}_{t_{k}}^{j} + \mu( \overline{V}_{t_{k}}^{j}) h + \kappa^{\top} \Delta W_{t_{k}}^{j},$ \\
}
Let $N_{j} = N_{T}+1$. \\
\While{$(\overline{V}_{t_{N_{T}}}^{j} - v_{0})(\overline{V}_{t_{N_{j}}}^{j} - v_{0}) > 0$}{
Sample $\Delta W_{t_{k}}^{j}$ from a Gaussian vector. \\
$\overline{V}_{t_{N_{j}+1}}^{j} = \overline{V}_{t_{N_{j}}}^{j} + \mu( \overline{V}_{t_{k}}^{j}) h + \kappa^{\top} \Delta W_{t_{k}}^{j},$ \\
$N_{j} = N_{j}+1$
}
Set, $h N_{j} = \tilde{\tau_{j}}$.
}
\For{$m=0$, ..., $M$}{
\For{$j=1, ..., B$}{
Set, $\phi_{t_{-1}} = 0$. \\
\For{$k\in \BRA{0, ..., N_{j}-1}$, starting from $\overline{Y}_{0}^{j} = y_{0}$ }{
$ \psi_{t_{k}}^{\theta_{2}^{m}, \bar{\lambda}^{m}, j} =  h F(\overline{V}_{t_{k}}^{j}, \Z^{\theta^{m}}(\overline{V}_{t_{k}}^{j})) - \bar{\lambda}^{m} h - \Z^{\theta^{m}}(\overline{V}_{t_{k}}^{j})^{\top} \Delta W_{t_{k}},$ \\
$\phi_{t_{k}}^{\theta_{2}^{m}, \bar{\lambda}^{m}, j} = \phi_{t_{k-1}}^{\theta_{2}^{m}, \bar{\lambda}^{m}, j} + \psi_{t_{k}}^{\theta_{2}^{m}, \bar{\lambda^{m}}, j}$ \\
}}
\For{$k\in \BRA{0, ..., N_{j}-1}$}{
Define the set $\T_{k} = \BRA{j \in \BRA{1, ..., B} \, ; \, \bar{\tau_{j}}^{+} \geq t_{k}}$. \\
Compute $ L_{\loc, t_{k}}^{B}(\theta_{1}^{m}, \theta_{2}^{m}, \bar{\lambda}^{m}) = \frac{1}{\ABS{\T_{k}}} \sum_{j \in \T_{k}} \ABS{\Y^{\theta_{1}^{m}}(\overline{V}_{t_{k}}^{j}) + \phi_{t_{k}}^{\theta_{2}^{m}, \bar{\lambda}^{m}, j} - y_{0}}^{2}$. 
}
Compute $ L_{\loc}^{B}(\theta_{1}^{m}, \theta_{2}^{m}, \bar{\lambda}^{m}) = \displaystyle\sum_{k=1}^{\max_{j=1, .., B} N_{j} - 1} L_{\loc, t_{k}}^{B}(\theta_{1}^{m}, \theta_{2}^{m}, \bar{\lambda}^{m})$. \\
Denoting $\theta = (\theta_{1}, \theta_{2}, \bar{\lambda})$, update $\theta^{m+1} = \theta^{m} - \rho_{m} \nabla_{\theta} L_{\loc}^{B}(\theta^{m})$.
}
\caption{Locally additive eBSDE Algorithm - (LAeBSDE)}
\label{locebsdealg}
\end{algorithm}

\paragraph*{}
For the simulation, we consider as before neural networks with $2$ hidden layers of $20 + d$ neurons each. We check that increasing the number of layers or neurons does not improve accuracy in our numerical tests. The algorithms are implemented in Python with \textit{Tensorflow} library. Numerical experiments are conducted using Intel(R) Xeon(R) CPU @ 2.20GHz with 25GB of RAM. The code of both solvers is available on github : \url{https://github.com/gubrx/Deep-learning-eBSDE}.

\begin{rmq}
Note that this algorithm can also be used by first approximating the ergodic cost $\lambda$ with Monte Carlo methods using the results of Section \ref{section:semiexplicit}, and plugging this estimator $\hat{\lambda}$ in the forward discretization above. The optimization is then only performed on the parameters of the neural network.
\end{rmq}

\subsection{Examples}

In this section, we present the numerical results obtained with Algorithm \ref{debsdealg} and \ref{locebsdealg} for two examples of ergodic BSDEs with explicit solutions. We, also, investigate the approximation of the ergodic cost $\lambda$ with Monte Carlo methods, within the framework of Propositions \ref{linlambd} and \ref{linquadr}. For the numerical tests, we will consider a stochastic factor $V$ of type Ornstein-Uhlenbeck with dynamics
\begin{eqnarray} \label{OU:def}
dV_{t} = - \mu V_{t} + \kappa^{\top}dW_{t}, \quad V_{0} = v_{0}.
\end{eqnarray}

\begin{exe} \label{exe1}
The ergodic BSDE \eqref{ebsde} with driver $F(v, z) = C_{v} v e^{-v^{2}/2}$ admits a unique Markovian solution such that $y(0) = \frac{C_{v}}{\mu + \frac{1}{2} \kappa^{2}} \frac{\sqrt{2 \pi}}{2}$ and $z$ is bounded, given by
\begin{eqnarray} \label{ex1}
\PAR{y(v), z(v), \lambda} = \PAR{ \frac{C_{v}}{\mu + \frac{1}{2} \kappa^{2}} \int_{-\infty}^{v} e^{-\frac{y^{2}}{2}}dy, \, \frac{C_{v}}{\mu + \frac{1}{2}\kappa^{2}} e^{-\frac{v^{2}}{2}}, \, 0}.
\end{eqnarray}
\end{exe}

Consider a discretization with step $h = 0.01$, a batchsize $B=64$ and an initial learning rate $\rho_{0} = 0.0003$. For the simulation, we choose $v_{0} = 0$, $C_{v} = 1, \, \kappa = 0.8$, $\mu = 1.5$ and $T = 1.$ In Figure \ref{lossex1} and \ref{lambdaex1}, we plot the evolution of the empirical loss function $L^{B_{\epsilon}}$ given in \eqref{emplossfunc} over $B_{\epsilon} = 100B$ samples as well as the absolute error on $\bar{\lambda}$ through the number $M$ of gradient descent performed in the algorithm.

\begin{figure}[H]
\centering
\begin{minipage}[t]{0.48\textwidth}
        \includegraphics[scale=0.5]{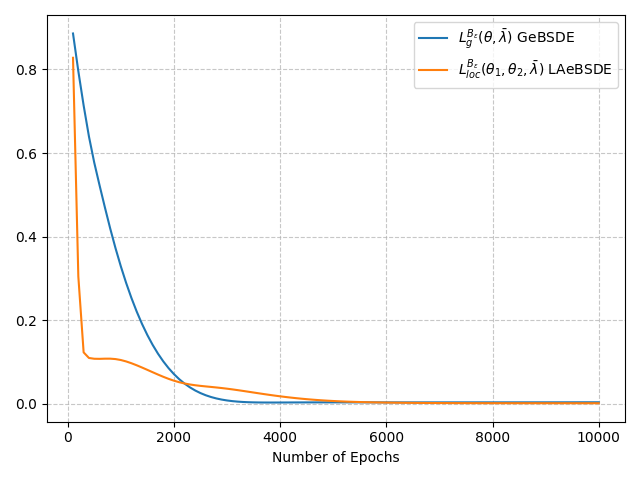}
        \caption{Empirical loss function $L^{B_{\epsilon}}$.}
    \label{lossex1}
\end{minipage} \hfill
\begin{minipage}[t]{0.48\textwidth}
        \includegraphics[scale=0.5]{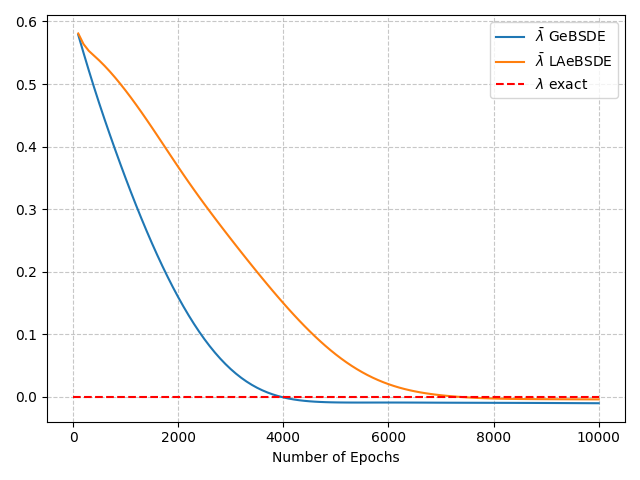}
        \caption{Convergence of $\lambda$.}
    \label{lambdaex1}
\end{minipage}
\end{figure}

Both algorithms converge in the sense that the loss functions as well as the absolute errors on $\lambda$ converge to zero as the number of training steps grows. The GeBSDE converges faster, in $4000$ gradient descent to the true value of $\lambda$ leading to an error of order $10^{-2}.$ The LAeBSDE estimation of the ergodic cost converges around $7000$ epochs with a final error of order $10^{-3}$. In Figure \ref{mean_err_ex1}, we plot the the mean relative absolute error on $Y$, defined by  
\begin{eqnarray}\label{empmeanerr}
    \epsilon_{t_{i}}(Y) = \frac{1}{B_{\epsilon}} \sum_{j=1}^{B_{\epsilon}} \ABS{\frac{y(\overline{V_{t_{i}}^{j}}) - \overline{Y_{t_{i}}^{j}}}{y(\overline{V_{t_{i}}^{j}})}},
\end{eqnarray}
at each time step that is for a sample of size $B_{\epsilon} = 100 B$ of realizations of the diffusion and $t_{i} \in \pi \cap \SBRA{0, T}$.

\begin{figure}[H]
\centering
\begin{minipage}[t]{0.6\textwidth}
    \centering
    \includegraphics[scale=0.5]{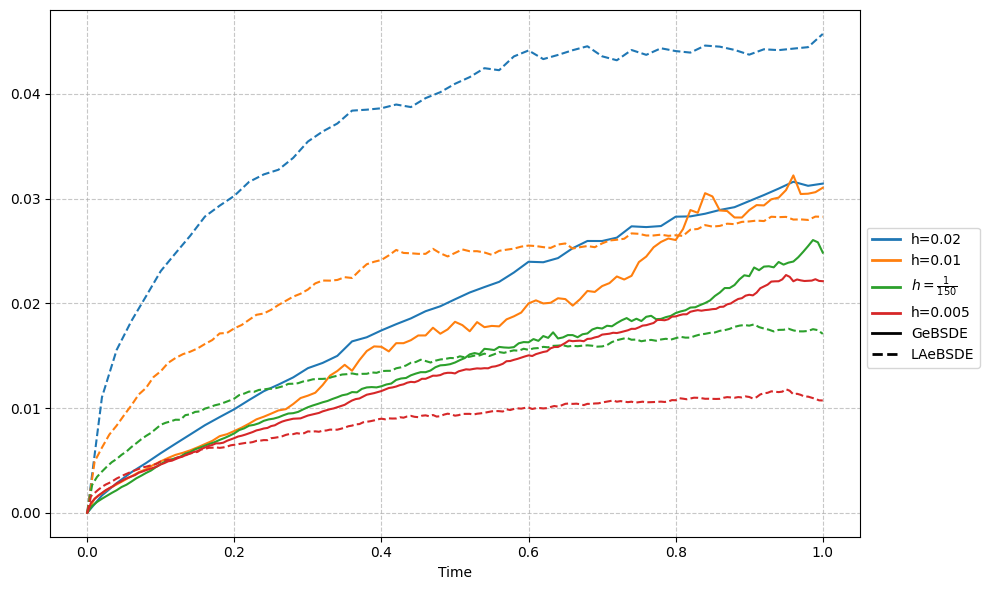}
    \caption{Mean relative error \eqref{empmeanerr} on $Y$ for different time steps.}
    \label{mean_err_ex1}
\end{minipage}
\end{figure}
Since the initial value $Y_{0} = y_{0}$ is known, we expect that the error is zero at time $0$ and increases on $\SBRA{0, T}$, and this can be seen in Figure \ref{mean_err_ex1}. The mean relative error also decreases as the time step gets smaller for both algorithms. Note that the GeBSDE leads a better mean relative error for $h=0.02$, $h=0.01$. For smaller time steps, the LAeBSDE algorithm outperforms the GeBSDE and leads a relative error of $1 \%$ at time $T = 1$ for $h=0.005$.

We also evaluate the error on $Y$ and $Z$ along the trajectories on $\SBRA{0, T}$ through the integral errors:
\begin{eqnarray} \label{def:integralerr}
I_{\epsilon}^{h}(Y) = \e \SBRA{ \sum_{i=1}^{N_{T}} h \ABS{y(\overline{V_{t_{i}}^{j}}) - \overline{Y_{t_{i}}^{j}}}} \quad \text{and} \quad
I_{\epsilon}^{h}(Z) = \e \SBRA{\sum_{i=1}^{N_{T}} h \NRM{z(\overline{V_{t_{i}}^{j}}) - \Z^{\theta}(\overline{V_{t_{i}}^{j}})}^{2}}.
\end{eqnarray}
The expectations above are computed on a sample of size $B_{\epsilon}.$ Moreover, we represent the mean and $95 \%$ confidence interval over $5$ independent training procedures. The errors are computed for the same values of time step as in Figure \ref{mean_err_ex1}.

\begin{figure}[H]
\centering
\begin{minipage}[t]{0.48\textwidth}
    \centering
    \includegraphics[scale=0.5]{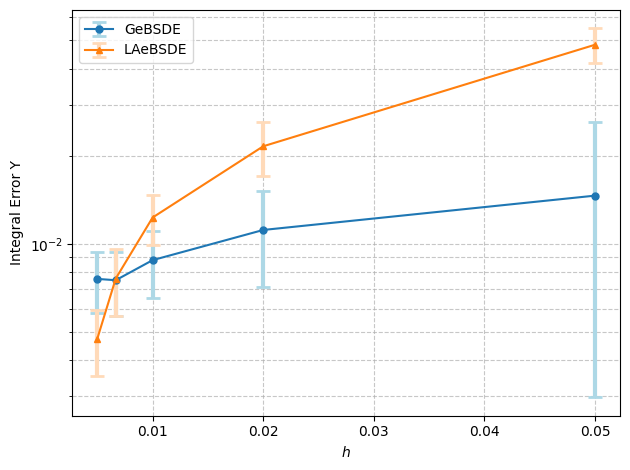}
    \caption{Integral error \eqref{def:integralerr} on $Y$ over $[0, T]$.}
    \label{int_err_Y_ex1}
\end{minipage}\hfill
\begin{minipage}[t]{0.48\textwidth}
    \centering
    \includegraphics[scale=0.5]{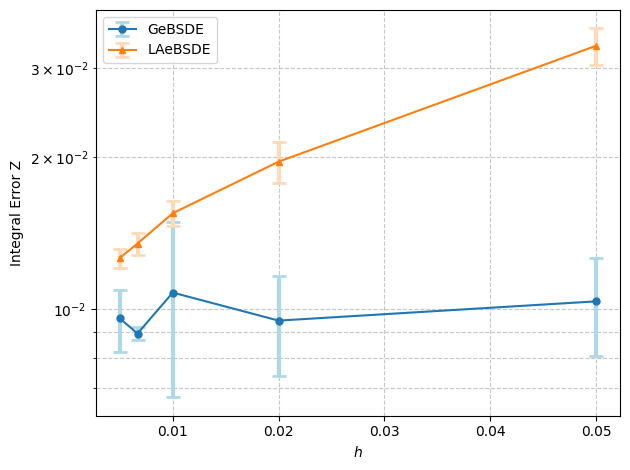}
    \caption{Integral error \eqref{def:integralerr} on $Z$ over $[0, T]$.}
    \label{inetgralerrorex1}
\end{minipage}
\end{figure}

The GeBSDE algorithm leads to a smaller integral error on $Z$ for every time steps. However, we again observe that the integral error on $Y$ for the LAeBSDE decreases rapidly with the time step and outperforms the global algorithm for $h=0.005$. Both algorithm provide a good approximation of the trajectory of the solution over the random interval $\SBRA{0, \tau}$, as displayed in Figure \ref{traj_ex1}.

\begin{figure}[H] 
\centering
\begin{minipage}{0.5\textwidth}
        \includegraphics[scale=0.6]{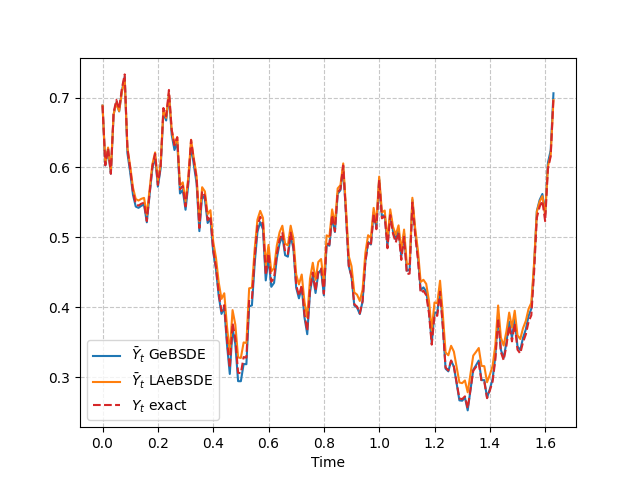}
        \caption{Example of trajectory of $Y$ over $\SBRA{0, \tau}$ with $h=0.01$.}
        \label{traj_ex1}
\end{minipage}
\centering
\end{figure}

\begin{exe}[with non-zero ergodic cost $\lambda$] \label{exe2}
We generalize the second example presented in \cite{hu2020systems}. Consider the ergodic BSDE \eqref{ebsde} with driver $F(v, z) = C_{v} \ABS{v} e^{-v^{2}/2}$. Denote $\Phi$ the cumulative distribution function of the standard normal distribution, $\Phi(x) = \frac{1}{\sqrt{2 \pi}} \int_{-\infty}^{x} e^{- \frac{y^{2}}{2}}dy$. 
\end{exe}

\begin{lem} \label{lemex2}
The eBSDE \eqref{ebsde} with generator $F(v, z) = C_{v} \ABS{v} e^{-v^{2}/2}$ admits a unique Markovian solution satisfying $y(0) = 0$ and such that $z$ is bounded, given by the following triplet $(y(.), z(.), \lambda)$
\begin{eqnarray}
y(v) &=& \ind_{\BRA{v \geq 0}} \int_{0}^{v} e^{\frac{y^{2}}{2}} \PAR{\frac{C_{v}}{\kappa^{2}} e^{-y^{2}} + 2 \frac{C_{v}}{\kappa^{2}} (\Phi(y) - 1)}dy + \ind_{\BRA{v < 0}} \int_{0}^{v} e^{\frac{y^{2}}{2}} \PAR{-\frac{C_{v}}{\kappa^{2}} e^{-y^{2}} + 2 \frac{C_{v}}{\kappa^{2}} \Phi(y)}dy  \nonumber \\
z(v) &=& \ind_{\BRA{v \geq 0}} \kappa e^{\frac{v^{2}}{2}} \PAR{\frac{C_{v}}{\kappa^{2}} e^{-v^{2}} + 2 \frac{C_{v}}{\kappa^{2}} (\Phi(v) - 1)} + \ind_{\BRA{v < 0}} \kappa e^{\frac{v^{2}}{2}} \PAR{- \frac{C_{v}}{\kappa^{2}} e^{-v^{2}} + 2 \frac{C_{v}}{\kappa^{2}} \Phi(v)} \nonumber \\
\lambda &=& \frac{C_{v}}{\sqrt{2 \pi}}. \label{lambdex2}
\end{eqnarray}
\end{lem}

\begin{preuve}
We look for a triplet $(y(.), z(.), \lambda)$ solution of the ergodic BSDE \eqref{ebsde} with generator $F(v, z) = C_{v} \ABS{v} e^{-v^{2}/2}$, of the form
\begin{eqnarray}
y(v) &=& \ind_{\BRA{v \geq 0}} \int_{0}^{v} e^{\frac{y^{2}}{2}} \PAR{A_{1} e^{-y^{2}} + A_{2}(\Phi(y) - 1)}dy + \ind_{\BRA{v < 0}} \int_{0}^{v} e^{\frac{y^{2}}{2}} \PAR{-A_{1} e^{-y^{2}} + A_{2} \Phi(y)}dy \nonumber\\
\label{ex2soly}  \\
z(v) &=& \ind_{\BRA{v \geq 0}} \kappa e^{\frac{v^{2}}{2}} \PAR{A_{1} e^{-v^{2}} + A_{2}(\Phi(v) - 1)} + \ind_{\BRA{v < 0}} \kappa e^{\frac{v^{2}}{2}} \PAR{- A_{1} e^{-v^{2}} + A_{2} \Phi(v)} \label{ex2solz} \\
\lambda &=& \frac{\kappa^{2} A_{2}}{2 \sqrt{2 \pi}}, \label{ex2sollamb}
\end{eqnarray}
where $A_{1}, \, A_{2}$ being real parameters to be determined.\\ An application of Ito formula gives
\begin{eqnarray*}
dy(V_{t}) &&= e^{V_{t}^{2} / 2} \SBRA{A_{1}e^{-V_{t}^{2}} + A_{2} (\Phi(V_{t}) - 1)} \PAR{- \mu V_{t} dt + \kappa dW_{t}} \\
&& \quad + \frac{1}{2} \kappa^{2} \PAR{V_{t}e^{V_{t}^{2}/2} \SBRA{A_{1}e^{-V_{t}^{2}} + A_{2} (\Phi(V_{t}) - 1)} + e^{V_{t}^{2}/2} \SBRA{-2 V_{t} A_{1}e^{-V_{t}^{2}} + \frac{A_{2}}{\sqrt{2 \pi}} e^{-V_{t}^{2}/2}}}dt \\
&&= V_{t}e^{V_{t}^{2}/2} \SBRA{A_{1}e^{-V_{t}^{2}} + A_{2} (\Phi(V_{t}) - 1)} \PAR{\frac{1}{2} \kappa^{2} - \mu}dt \\
&& \quad - \kappa^{2} A_{1} V_{t} e^{-V_{t}^{2}/2}dt + \frac{A_{2} \kappa^{2}}{2 \sqrt{2 \pi}}dt + \kappa e^{V_{t}^{2}/2} \SBRA{A_{1}e^{-V_{t}^{2}} + A_{2} (\Phi(V_{t}) - 1)} dW_{t}.   
\end{eqnarray*}
Choosing $A_{1} = \frac{C_{v}}{\kappa^{2}}$ and $\mu = \frac{1}{2} \kappa^{2}$, the triplet $(y, z, \frac{C_{v}}{\sqrt{2 \pi}})$ given by \eqref{ex2soly} and \eqref{ex2solz} satisfies equation \eqref{ebsde}. Now, we need to establish the Markovian property of the solution (see Proposition $3.4$ in \cite{liang2017representation}), namely that $y$ is $C^{2}$ and that $z(v) = \kappa \nabla y(v)$, so that $z$ is $C^{1}$. For this purpose, and thanks to the continuity of $z$ in $0$ and from \eqref{ex2solz} we obtain that $A_{2} = 2 A_{1} = 2 \frac{C_{v}}{\kappa^{2}}.$ Hence, we deduce the boundedness of $z(.)$ by $Z_{\max} = \kappa \frac{C_{v}}{\mu - C_{v}}$ and we have the Markovian property of the solution.
\end{preuve}

Let $v_{0} = 0$, $T=1$, $C_{v} = 0.75$, $\mu = 1$, time step $h = 0.01$ and the same parameters of the neural network as in Example \ref{exe1}. In this setting, the ergodic cost $\lambda$ given by \eqref{lambdex2} is $0.299206$ and the trainable parameter $\overline{\lambda}$ for both algorithms converges towards this value in approximately $6000$ gradient steps. Training the model with $10000$ gradient descent, the absolute error on $\lambda$ is of order $10^{-3}$. We illustrate the convergence of the empirical loss functions as well as the convergence of the ergodic cost estimators.

\begin{figure}[H]
\centering
\begin{minipage}[t]{0.48\textwidth}
    \includegraphics[scale=0.5]{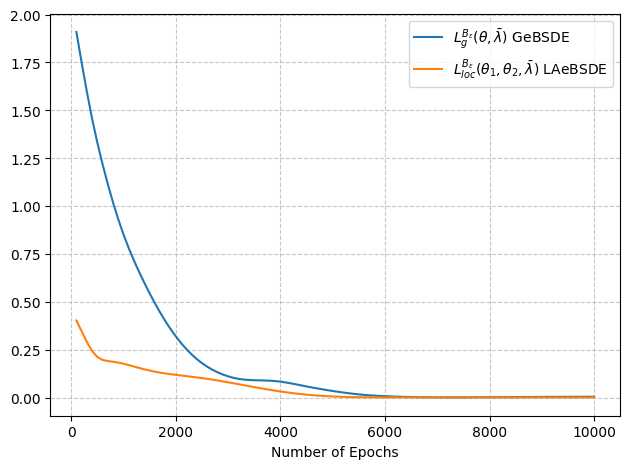}
    \caption{Empirical loss function $L^{B_{\epsilon}}$.}
\end{minipage}\hfill
\begin{minipage}[t]{0.48\textwidth}
    \includegraphics[scale=0.5]{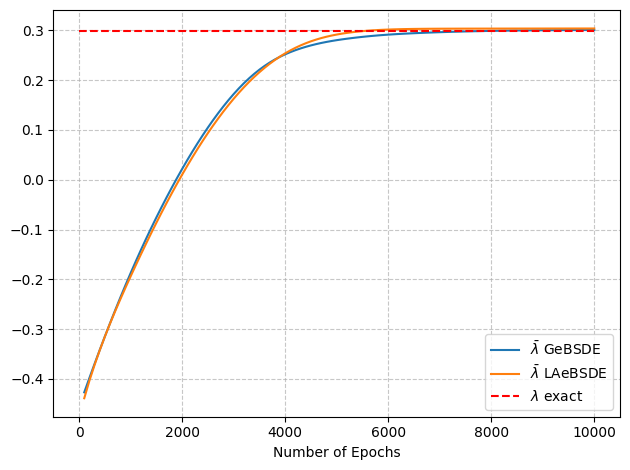}
    \caption{Convergence of $\bar{\lambda}$.}
\end{minipage}
\end{figure}

The shape of the mean absolute error on $Y$ for this example is quite different for the two algorithms. In fact, for the GeBSDE, where the solution $Y$ is constructed with a forward iterative scheme \eqref{discrebsde} relying on the trained neural network $\Z^{\theta}$, the error starts at zero and then grows almost linearly, according to Figure \ref{mean_err_ex2}. On the other hand, for the LAeBSDE, the solution $Y$ is the output of the neural network $\Y^{\theta_{1}}$, optimized according to the aggregation of local loss functions $L_{\loc, t_{i}}$ given in \eqref{loc_loss}. We observe that, for this example, the mean error is almost constant on the interval $\SBRA{0.2, T}$.

\begin{figure}[H]
\centering
\begin{minipage}[t]{0.6\textwidth}
    \centering
    \includegraphics[scale=0.5]{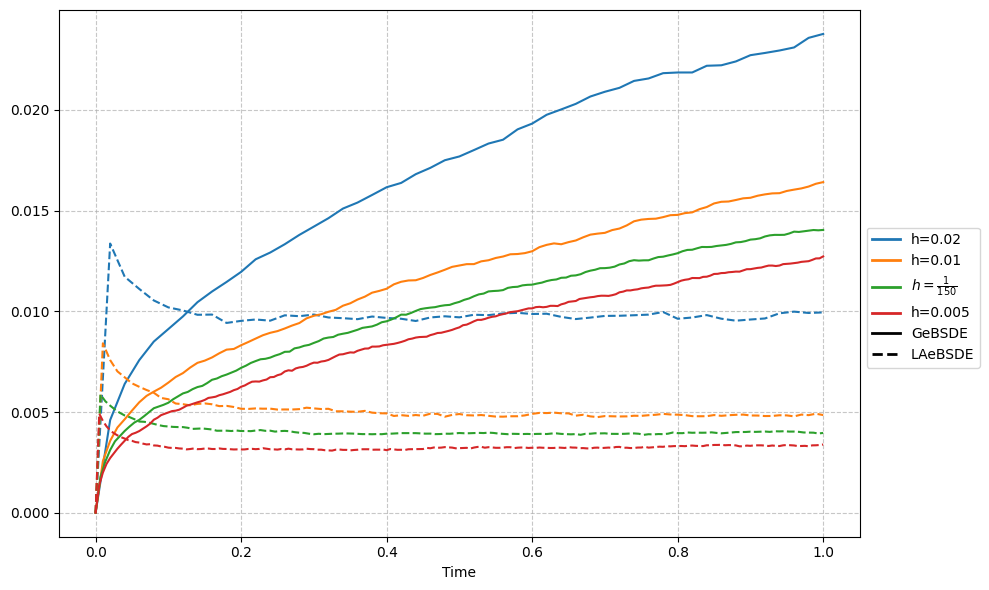}
    \caption{Mean relative error \eqref{empmeanerr} on $Y$ for different time steps.}
    \label{mean_err_ex2}
\end{minipage}
\end{figure}

As shown in Figure \ref{mean_err_ex2} and Figure \ref{int_err_Y_ex2}, the mean relative error and the integral error on $Y$ are lower and decreases faster for the LAeBSDE.

\begin{figure}[H]
\centering
\begin{minipage}[t]{0.48\textwidth}
    \centering
    \includegraphics[scale=0.5]{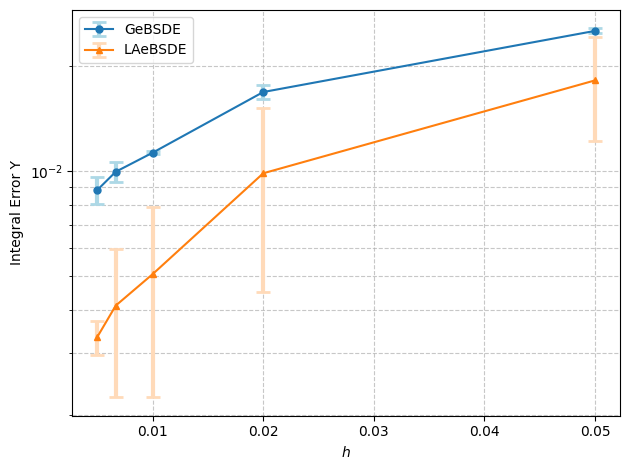}
    \caption{Integral error \eqref{def:integralerr} on $Y$ over $[0, T]$.}
    \label{int_err_Y_ex2}
\end{minipage}\hfill
\begin{minipage}[t]{0.48\textwidth}
    \centering
    \includegraphics[scale=0.5]{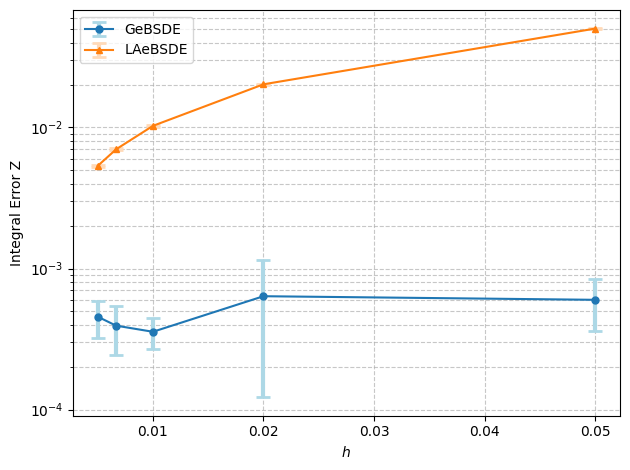}
    \caption{Integral error \eqref{def:integralerr} on $Z$ over $[0, T]$.}
    \label{integralerrorex2}
\end{minipage}
\end{figure}

\paragraph{Estimation of $\lambda$ with Monte Carlo methods -} Since for the two above examples the generator only depends on $v$, Proposition \ref{linlambd} applies and we recall the characterization of the ergodic cost $\lambda$ given in \eqref{MClambda1},
\begin{eqnarray*}
\lambda = \frac{\e \SBRA{\displaystyle\int_{0}^{\tau} F(V_{s}) ds}}{\e \SBRA{\tau}}.
\end{eqnarray*}
We use Monte Carlo methods over $M$ samples to compute both expectations in the above formula. Approximating $V$ and $\Gamma$ with an Euler scheme on the time grid $\pi$ with time step $h$ and the integral with a Riemann sum on the time grid $\pi$, we consider the estimate
\begin{eqnarray} \label{estimMC}
\hat{\lambda} = \frac{1}{\displaystyle\sum_{m=1}^{M} \bar{\tau}_{m}} \sum_{m=1}^{M} \sum_{i=0}^{\bar{\tau}_{m} -1} h F(\overline{V_{t_{i}}^{m}}).
\end{eqnarray}
We summarize the approximation results of the mean absolute error on $100$ simulations for the estimation of the ergodic cost $\hat{\lambda}$ for Examples \ref{exe1} and \ref{exe2} in Table \ref{tab:absolute_error_MCex1} and \ref{tab:absolute_error_MCex2} respectively. The simulations are performed with the following set of parameters : $T = 1$, $\mu = 2$, $\kappa = 2$, $C_{v} = 1$ and $v_{0} = 0.5$.

\begin{table}[H]
\centering
\begin{tabular}{|c|ccc|}
\hline
$h$ &M=1000 & M=10000 & M=100000 \\ \hline
0.05 & 0.009743 \, (4.87e-05) & 0.007224 \, (7.48e-06) & 0.005691 \, (1.37e-06) \\
0.02 & 0.009107 \, (4.66e-05) & 0.006597 \, (4.63e-06) & 0.004895 \, (1.69e-06) \\
0.01 & 0.008833 \, (3.73e-05) & 0.005778 \, (7.85e-06) & 0.004374 \, (9.73e-07) \\ \hline
\end{tabular}
\caption{Mean absolute error (variance) on $\hat{\lambda}$ for Example \ref{exe1}. The exact value of $\lambda$ is $0$.}
\label{tab:absolute_error_MCex1}
\end{table}

\begin{table}[H]
\centering
\begin{tabular}{|c|ccc|}
\hline
$h$ & M=1000 & M=10000 & M=100000 \\ \hline
0.05 & 0.002988 \, (2.86e-06) & 0.002955 \, (2.50e-07) & 0.002904 \, (1.99e-08) \\
0.02 & 0.002136 \, (2.30e-06) & 0.001617 \, (4.98e-07) & 0.001528 \,(1.67e-08) \\
0.01 & 0.001637 \, (1.29e-06) & 0.000780 \, (2.62e-07) & 0.000939 \, (3.55e-08) \\ \hline
\end{tabular}
\caption{Mean absolute error (variance) on $\hat{\lambda}$ for Example \ref{exe2}. The exact value of $\lambda$ given in \eqref{lambdex2} is $0.398942$.}
\label{tab:absolute_error_MCex2}
\end{table}

Finally, we display the mean and variance of lambda estimations obtained with Monte Carlo methods \eqref{estimMC} using $M=10 0000$ samples with the output of Algorithm \ref{debsdealg} and \ref{locebsdealg} for $B=64$ and $10 000$ gradient descents in Table \ref{tab:lambda_comp}. Statistics are computed over $100$ values for the Monte Carlo estimator and on $10$ independent trainings of our neural network algorithms.

\begin{table}[H]
\centering
\begin{tabular}{|c|cccc|}
\hline
& Exact & MC & GeBSDE & LAeBSDE \\ \hline
Example \ref{exe1} & 0 & -0.004374 \, (9.73e-07) & -0.003782 \, (4.28e-05) & -0.004280 \, (3.07e-05) \\
Example \ref{exe2} & 0.398942 & 0.399882  \, (3.55e-08) & 0.400130 \, (1.53e-05) & 0.397600 \, (4.63e-05) \\ \hline
\end{tabular}
\caption{Comparison of $\lambda$ approximations for parameters $v_{0} = 0.5$, $T= 1, \, h=0.01, \, \kappa= 2,  \, \mu=2, \, C_{v}=1$.}
\label{tab:lambda_comp}
\end{table}

\subsection{Power utility examples}

We now revert our attention to ergodic BSDEs associated with power forward utilities \eqref{powut}. In the absence of portfolio constraints, the generator \eqref{driver_zar} can be rewritten for $(v, z) \in \R \times \R^{d}$, as
\begin{eqnarray*}
   F^{\delta}(v, z) = \frac{\delta}{2(1 - \delta)} \NRM{\theta(v) + z}^{2} + \frac{1}{2}\NRM{z}^{2}.
\end{eqnarray*}
Corollary \ref{linquadr} gives a characterization of the ergodic cost $\lambda$ as the solution to the minimization problem \eqref{minlambda}, which we will use as a benchmark for the ergodic cost. 
First, we approximate the diffusion $V$ and the process $\Gamma$ with an Euler scheme and the expectation with Monte Carlo method. Then, we obtain the approximation of the cost $\lambda$ by using Newton's method for the minimization on $\SBRA{-K, K}$ of the map
\begin{eqnarray} \label{MCminlambd}
    \lambda \mapsto \ABS{\frac{1}{M} \sum_{m=1}^{M} \overline{\Gamma}_{0, \bar{\tau}} -1}.
\end{eqnarray}

For the simulation, we consider a truncated linear price of risk vector $\theta(v) = \varphi_{b}(\theta v)$, as performed in \cite{pham2021neural}, where $\varphi_{b}$ denotes the projection on the segment $\SBRA{-b ,b}$. We use the same stochastic factor $V$ of type Ornstein-Uhlenbeck as in \eqref{OU:def} and set the parameters $\mu = 3$, $\kappa = 1.3$, $\delta = 0.5$, $\theta=0.8$, $b=3$ and $T = 1.$ The number of Monte Carlo samples used for the estimation of the ergodic cost with \eqref{estimMC} is $M=100 000$ and we use a time step $h = 0.01.$ Finally, we observe that the bound $K$ is larger for the following than for the examples from the previous section. Thus we use a higher learning rate of $\rho_{0} = 0.0007$ to ensure a sufficient speed of convergence for $\bar{\lambda}$.

\begin{figure}[H]
\centering
\begin{minipage}[t]{0.48\textwidth}
        \includegraphics[scale=0.5]{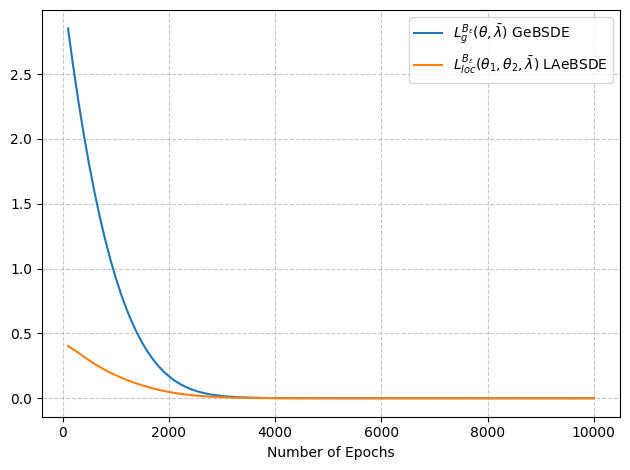}
        \caption{Empirical loss function $L^{B^{\epsilon}}$.}
\end{minipage} \hfill
\begin{minipage}[t]{0.48\textwidth}
        \includegraphics[scale=0.5]{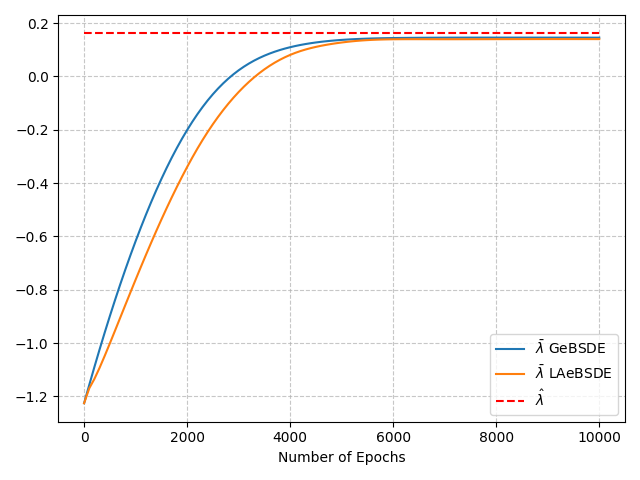}
        \caption{ Convergence of $\bar{\lambda}$.}
\end{minipage}
\end{figure}

The ergodic cost $\bar{\lambda}$ computed with Algorithm \ref{debsdealg} and \ref{locebsdealg} and the Monte Carlo approximation $\hat{\lambda}$ converge towards the same value. The absolute error between the two types of estimators is of order $10^{-2}$. Finally we illustrate the convergence of the estimator given by \eqref{MCminlambd} depending on the time step $h$ and the number of Monte Carlo samples $M$.

\begin{table}[h]
\centering
\begin{tabular}{|c|ccc|}
\hline
$h$ & M=1000 & M=10000 & M=100000 \\ \hline
0.10 & 0.192683 \, (3.48e-04) & 0.201704 \, (5.5e-05) & 0.206999 \, (3.9e-05) \\
0.05 & 0.187675 \, (8.03e-04) & 0.183987 \, (2.17e-04) & 0.182469 \, (2.6e-05) \\
0.02 & 0.159276 \, (1.55e-03) & 0.173734 \, (4.17e-04) & 0.173557 \, (1.31e-04) \\
0.01 & 0.159119 \, (1.16e-03) & 0.174493 \, (1.61e-03) & 0.169749 \, (1.8e-04) \\ \hline
\end{tabular}
\caption{Mean (variance) on $\bar{\lambda}$ on $10$ independent runs.}
\label{tab:mean_variance_results}
\end{table}

Those Monte Carlo approximations of $\lambda$ allow to use the semi-explicit representation of Section \ref{section:connection} in order to simulate the solution of the ergodic BSDE \eqref{ebsde} on $\SBRA{0, \tau}.$ However, in the general case of generator with quadratic growth for which the Cole-Hopf transform does not help to reduce to a linear BSDE, one need another approximation procedure of the ergodic cost $\lambda$.

\paragraph{A two dimensional example -} Consider a financial market consisting of one stock, whose price dynamics is given by
\begin{eqnarray*}
    dS_{t} = S_{t} \PAR{b(V_{t})dt + \sigma(V_{t})dW_{t}^{1}},
\end{eqnarray*}
with the stochastic factor satisfying
\begin{eqnarray*}
dV_{t}^{1} = \mu(V_{t})dt + \kappa_{1}dW_{t}^{2} + \kappa_{2} dW_{t}^{2}, \quad dV_{t}^{2} = 0.
\end{eqnarray*}
The admissible set of strategies is, thus, restricted to $\Pi = \R \times \BRA{0}$, so that $\pi_{t}^{2} = 0$ and the wealth equation \eqref{wealth} reduces to
\begin{eqnarray*}
    dX_{t}^{\pi} = X_{t}^{\pi} \pi_{t}^{1} \PAR{\theta(V_{t}) + dW_{t}^{1}}, \quad \text{with} \quad \theta(V_{t}) = \frac{b(V_{t})}{\sigma(V_{t})}.
\end{eqnarray*}
The generator \eqref{driver_zar} is then given by
\begin{eqnarray}
    F(V_{t}, Z_{t}) = \frac{1}{2} \frac{\delta}{1 - \delta} \ABS{Z_{t}^{1} + \theta(V_{t})}^{2} + \frac{1}{2} \NRM{Z_{t}}^{2}.
\end{eqnarray}
Denoting $\hat{\delta} = \frac{1 - \delta + \delta (\frac{\kappa^{1}}{\NRM{\kappa}})^{2}}{1 - \delta}$ and $\tilde{Y_{t}} = e^{\hat{\delta}(Y_{t} - \lambda t)}$, the authors in \cite{liang2017representation} show that the function $\tilde{y}$ must satisfy
\begin{eqnarray} \label{edptwodim}
\tilde{y}_{t}(v, t) + \frac{1}{2} (\kappa_{1}^{2} + \kappa_{2}^{2}) \tilde{y}_{vv}(v, t) + \PAR{\mu(v) + \frac{\delta \kappa^{1}}{1 - \delta} \theta(v)} \tilde{y}_{v}(v, t) + \frac{\hat{\delta} \delta}{2(1 - \delta)} \theta^{2}(v) \tilde{y}(v, t) = 0.
\end{eqnarray}
Assuming a linear market price of risk $\theta(v) = \theta v$ and an Ornstein-Uhlenbeck stochastic factor with $\mu(v) = - \mu v$ and $\mu > 0,$ following the methodology of \cite{nadtochiy2014class}, solutions to this PDE can be derived.

\begin{figure}[H]
\centering
\begin{minipage}[t]{0.48\textwidth}
        \includegraphics[scale=0.5]{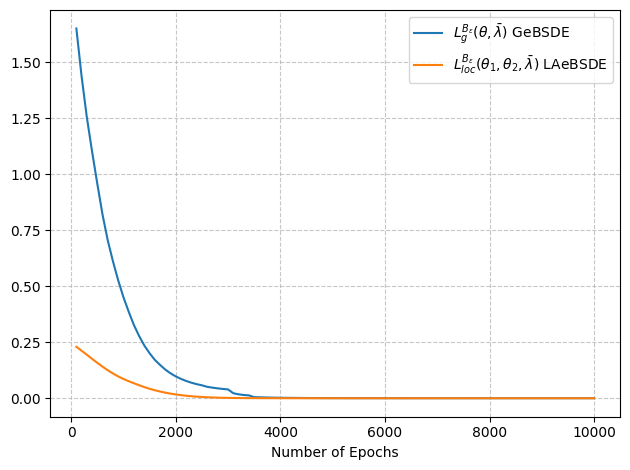}
        \caption{Empirical loss function $L^{B^{\epsilon}}$.}
\end{minipage} \hfill
\begin{minipage}[t]{0.48\textwidth}
        \includegraphics[scale=0.5]{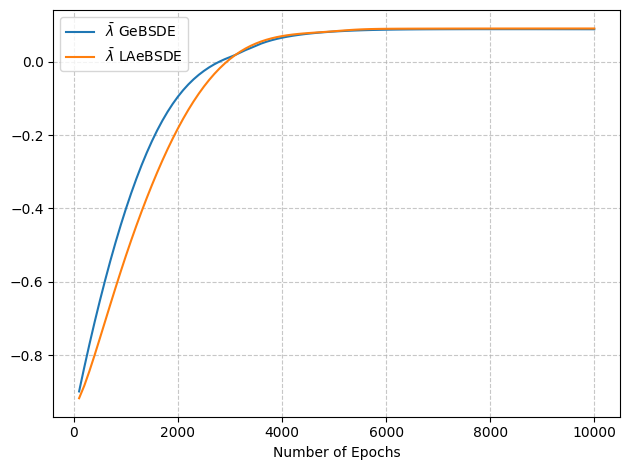}
        \caption{Convergence of $\bar{\lambda}$.}
\end{minipage}
\end{figure}

Validation loss functions for both algorithms converge to zero and the trainable parameters $\bar{\lambda}$ also converge towards the same value.

\paragraph*{}
Finally, we revert to our initial objective to simulate homothetic forward progressive utilities of Section \ref{section:forwardutilities}. Using one of the two algorithms GeBSDE or LAeBSDE to simulate the solution of eBSDE \eqref{ebsde} on $\SBRA{0, T}$, we are now able to plot the corresponding forward utilities. We display the shape of the approximated random field $U$ given by \eqref{powut} for one realization of the diffusion $V$. One can also access the rescaled optimal portfolio $\pi_{t}^{*}$ given by \eqref{pi_pow}. We plot the example of trajectory associated to the same realization of this power utility in Figure \ref{opt_port_plot}.

\begin{figure}[H]
\centering
\begin{minipage}[t]{0.48\textwidth}
        \includegraphics[scale=0.5]{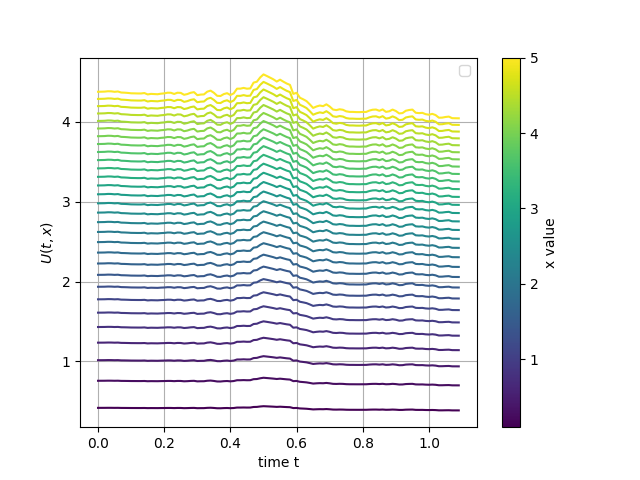}
        \caption{Dynamics of approximated utility $U(t, x)$.}
\end{minipage} \hfill
\begin{minipage}[t]{0.48\textwidth}
        \includegraphics[scale=0.5]{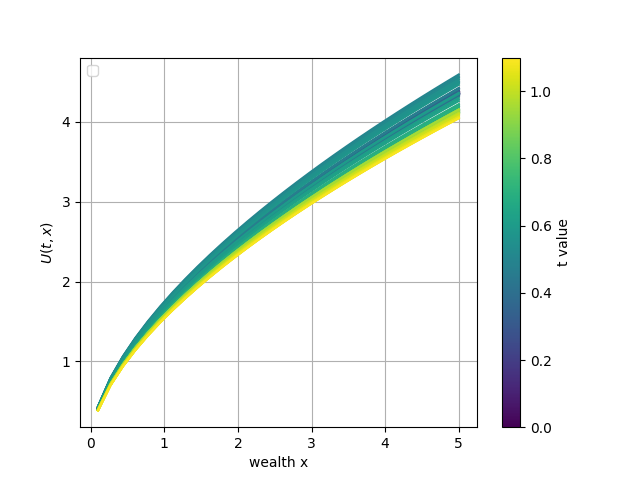}
        \caption{Monotonicity and concavity of approximated utility $U(t,x)$.}
\end{minipage}
\end{figure}

\begin{figure}[H]
\centering
\begin{minipage}[t]{0.48\textwidth}
        \centering
        \includegraphics[scale=0.5]{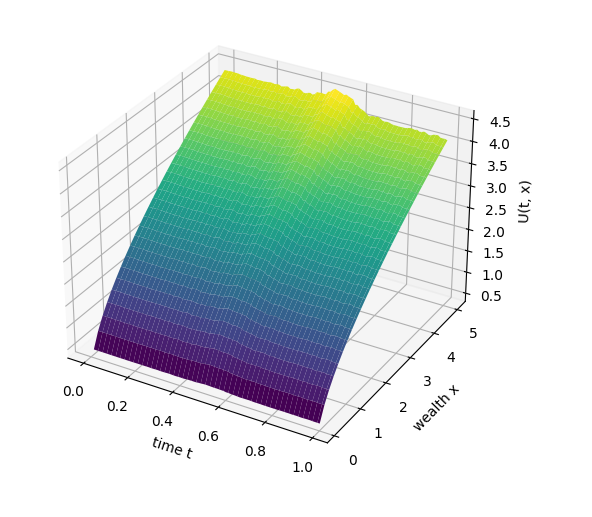}
        \caption{Random field $U(t, x)$.}
\end{minipage} \hfill
\begin{minipage}[t]{0.48\textwidth}
        \centering
        \includegraphics[scale=0.5]{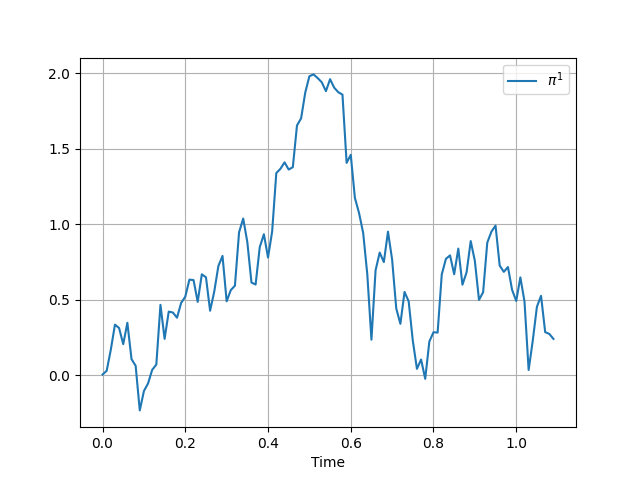}
        \caption{Rescaled optimal strategy $\pi_{t}^{1}$.}
        \label{opt_port_plot}
\end{minipage}
\end{figure}

\paragraph*{Acknowledgements} The authors thank Z. Bensaid (LMM - Le Mans University) for helpful discussions on deep learning methods for the simulation of BSDEs. The authors thank the anonymous referees and the associate editor for their helpful comments.

\newpage

\appendix

\newpage
\bibliographystyle{plain}
\bibliography{Biblio}

\end{document}